\documentclass[a4paper,10pt]{article}
\usepackage[utf8]{inputenc}
\usepackage{commands,geometry}
\usepackage{authblk}
\usepackage{xcolor}
\usepackage{appendix}

\usepackage{lmodern}
\title{Varieties of truth definitions}

\author[1]{Piotr Gruza}
\author[2]{Mateusz {\L}e{\l}yk}
\affil[1]{Doctoral School of Exact and Natural Sciences, University of Warsaw, mailto: p.gruza3@uw.edu.pl}
\affil[2]{Faculty of Philosophy, University of Warsaw, mailto: mlelyk@uw.edu.pl}

\begin{document}

\maketitle

\begin{abstract}
We study the structure of the partial order induced by the definability relation on definitions of truth for the language of arithmetic. Formally, a definition of truth is any sentence $\alpha$ which extends a weak arithmetical theory (which we take to be $\EA$) such that for some formula $\Theta$ and any arithmetical sentence $\varphi$, $\Theta(\qcr{\varphi})\equiv \varphi$ is provable in $\alpha$. We say that a sentence $\beta$ is definable in a sentence $\alpha$, if there exists an unrelativized translation from the language of $\beta$ to the language of $\alpha$ which is identity on the arithmetical symbols and such that the translation of $\beta$ is provable in $\alpha$. Our main result is that the structure consisting of truth definitions which are conservative over the basic arithmetical theory forms a countable universal distributive lattice. Additionally, we generalize the result of Pakhomov and Visser showing that the set of (G\"odel codes of) definitions of truth is not $\Sigma_2$-definable in the standard model of arithmetic. We conclude by remarking that no $\Sigma_2$-sentence, satisfying certain further natural conditions, can be a definition of truth for the language of arithmetic. 
\end{abstract}

\paragraph{Acknowledgement} This research was supported by the NCN SONATA grant "Truth: Between Disquotation and Compositionality." (2019/35/D/HS1/04502)
\section{Introduction}


The main goal of this paper is to understand in how many alternative ways can a definition of truth be given. More precisely: assuming that a first-order language $\mathcal{L}$ has been fixed, how many different truth definitions for $\mathcal{L}$ are there? The far reaching objective is to understand essential properties that any truth definition for $\mathcal{L}$ must have. In particular, we would like to understand whether the definitions of truth need to be compositional.\footnote{This question, along with its formal paraphrase, was suggested by Albert Visser (see \cite{viss19enayat})} We say more about this problem in Section \ref{sec_between_disq_and_comp}. 

\medskip
In what follows, we define a definition of truth for $\mathcal{L}$ to be any sentence $\alpha$, such that for some formula $\Theta$ of the language of $\alpha$, all Tarski biconditionals for $\mathcal{L}$ are provable in $\alpha$. More precisely, for $\alpha$ to be a truth definition, there must be a formula $\Theta(x)$, such that for every $\mathcal{L}$-sentence $\psi$, the sentence
\[\Theta(\qcr{\psi})\leftrightarrow\psi\]
is provable in $\alpha$. We specialize to the case when $\mathcal{L}$ is the usual language of first-order arithmetic, denoted by $\LA$. We assume that all the truth definitions extend a weak finitely axiomatised fragment of Peano Arithmetic, which provides a metatheory for formalising the syntax for $\LA$. The set of all truth definitions for $\LA$ is denoted by $\DEF$. For most of the paper, we are interested in the subset of $\DEF$ consisting of those truth definitions for $\LA$ which are (syntactically) conservative extensions of the basic arithmetical theory. The set of such truth definitions is denoted by $\cDEF$. To compare the definitions of truth, we consider the definability pre-order $\LHD$ on $\DEF$ and determine its various properties. Informally speaking, for two truth definitions $\alpha$, $\beta$, we put $\alpha\LHD\beta$ if there is a translation $\sigma$ from the language of $\alpha$ to the language of $\beta$ which preserves the symbols of $\mathcal{L}$ such that $\sigma[\alpha]$ is provable in $\beta$.

\medskip
The main result of this paper is that the structure $\langle \cDEF, \LHD\rangle$ is a countable universal distributive lattice, i.e. it is a distributive lattice which embeds every countable distributive lattice. The proof proceeds by constructing an embedding of the lattice structure of the countable atomless Boolean algebra into a sublattice of $\langle \cDEF, \LHD\rangle$, which consists uniquely of truth definitions, which are provable in the canonical truth theory $\CT^-$. Restricting to conservative truth definitions is required to make sure that the result is non-trivial, in particular that it does not immediately follow from the known facts about the Lindenbaum algebras of r.e. extensions of basic arithmetical systems. Thus, we take our main result to witness a great diversity among the possible definitions of truth. We provide some further information about the structure of $\langle\DEF, \LHD\rangle$ and include some results of the possible complexities of truth definitions and the connection between truth definitions and compositional Tarski clauses.

\medskip
There are at least three important predecessors to the current paper: in \cite{pakhomov_visser} it was proved that, in the order induced by the relation of definability, there are no minimal finite axiomatization (in an extended language) of a given non-finitely axiomatizable theory (such as $\PA$, ZF, $\ldots$). In \cite{viss19enayat} the same reasoning was applied to show that, in the order induced by definability, there are no minimal definitions of truth for arithmetic. \cite{enayat_lelyk_axiom} shows, among other things, that the ordering induced by provability on a very concrete set of natural finite axiomatizations (which carry truth defintitions for the language of arithmetic) of $\PA$ is a countable universal partial order.

\medskip
The structure of the paper is as follows: in Section 2 we introduce all the relevant preliminaries including the first-order arithmetic, axiomatic theories of truth, interpretability and (basic) lattice theory. In Section 3 we prove our main result. In Section 4 we provide some further details on the structure of the lattice $\langle \cDEF, \LHD\rangle$, in particular showing another way of realising the embedding of the countable atomless Boolean algebra into $\cDEF$. The image of the latter embedding consists uniquely of truth definitions which are \emph{semantically} conservative over the base arithmetical theory, that is every model of the base theory can be expanded to a model of any such truth definition. Last but not least, in this section we prove that no truth definition, which satisfies some further natural conditions, can be a $\Sigma_2$-sentence. We conclude in Section 5 with a number of open questions.


\section{Preliminaries}


\subsection{Arithmetical Preliminaries}\label{sec_arithm}


We take the language of arithmetic, denoted by $\LA$, to be the first-order with equality language of the signature $\langle0,1,+\times,<\rangle$, where $0$ and $1$ are constant symbols, $+$ and $\times$ are binary function symbols, and $<$ is a binary relation symbol. In the definition below we recall the standard definition of the arithmetical hierarchy, $\mathcal{L}$ being any language extending $\LA$:

\begin{definition}
A formula is \emph{atomic} iff it is of the form $R(t_{1}, ..., t_{n})$ or $t_{1}=t_{2}$, where $R$ is an $n$-ary relation symbol and $t_{1}$, $t_{2}$, ..., $t_{n}$ are terms.

\medskip

The class of \emph{quantifier-free} formulae of $\mathcal{L}$ is the smallest set containing all atomic formualae of $\mathcal{L}$ and closed under propositional connectives, i.e. if $\varphi$ and $\psi$ are quantifier-free, then so are $\lnot\varphi$ and $\varphi\land\psi$.

\medskip 

The class of $\Delta_{0}$ formulae of $\mathcal{L}$ is the smallest set containing all quantifier-free formulae of $\mathcal{L}$ and closed under the adjoin of bounded quantifiers, i.e. if $\varphi$ is a $\Delta_{0}$ formula, $x$ is a variable and $t$ is a term not containing $x$, then $\exists x(x<t\land\varphi)$ (abbreviated $\exists x<t. \varphi$) and $\forall x(x<t\rightarrow\varphi)$ (abbreviated $\forall x<t.\varphi$) are $\Delta_0$ as well.

\medskip 

The classes of $\Pi_{0}$ and $\Sigma_{0}$ formulae of $\mathcal{L}$ are exactly the class of $\Delta_{0}$ formulae of $\mathcal{L}$.

\vspace{0.2cm}
For a natural number $n$ greater than $0$, the class of $\Pi_{n}$ formulae of $\mathcal{L}$ is the smallest set containing all $\Sigma_{n-1}$ formulae of $\mathcal{L}$ and closed under adding the universal quantifier.

\vspace{0.2cm}
For a natural number $n$ greater than $0$, the class of $\Sigma_{n}$ formulae of $\mathcal{L}$ is the smallest set containing all $\Pi_{n-1}$ formulae of $\mathcal{L}$ and closed under adding the existential quantifier.

\vspace{0.2cm}
For a natural number $n$ and a theory $Th$, the class of $\Delta_{n}(Th)$ formulae of $\mathcal{L}$ consists of formulae $\varphi(\overline{x})\in\mathcal{L}$ for which there exist a $\Sigma_{n}$ formula $\sigma(\overline{x})\in\mathcal{L}$ and a $\Pi_{n}$ formula $\pi(\overline{x})\in\mathcal{L}$ such that $Th\vdash\forall\overline{x}(\varphi(\overline{x})\leftrightarrow\sigma(\overline{x}))$ and $Th\vdash\forall\overline{x}(\varphi(\overline{x})\leftrightarrow\pi(\overline{x}))$. In particular, $\Delta_{0}(Th)$ is exactly the set of formulae $Th$-provably equivalent to some $\Delta_{0}$ formulae.

\vspace{0.2cm}
Apart from the $\Sigma_n\setminus\Pi_n$ measures of complexity we employ also a simpler one. If $\varphi$ is any sentence, then $\dpt(\varphi)$ denotes its logical depth, i.e. maximal nesting of connectives and quantifiers. $\dpt(\varphi)$ is defined recursively: if $\varphi$ is atomic, then $\dpt(\varphi) = 0$, $\dpt(\neg\varphi) = \dpt(\exists z.\varphi) = \dpt(\varphi) +1$ and $\dpt(\varphi\wedge \psi) = \max\{\dpt(\varphi), \dpt(\psi)\}+1$.
\end{definition}

\begin{definition}
The theory $\PAm$ is the set of axioms of nonnegative parts of discrete ordered rings (see Section 2.1 of Chapter 2 of \cite{kaye}.)
\end{definition}

This theory is used as a basis for several arithmetical theories. In the following, the subsequent theories are formed by adding the instances of the parameter induction scheme for the subsequent classes of the arithmetical hierarchy.

\begin{definition}
Given a natural number $n$, the theory $\IPn$ is $\PAm$ augmented with the instances of the  induction scheme (with parameters) for $\Pi_{n}$ formulae of $\LA$. Likewise for $\ISn$ and $\IDz$. Peano Arithmetic, $\PA$, is $\PAm$ together with all ($\LA$) instances of parameter induction scheme added. If $\mathcal{L}$ is any language (extending the language of arithmetic), then $\IND(\mathcal{L})$ denotes the collection of all induction axioms for formulae of $\mathcal{L}$.
\end{definition}

The weakest system we are interested in lies in between $\IDz$ and I$\Sigma_1$.

\begin{definition}\label{defi_exp}
We use symbol $\exp$ to denote arithmetical representation of an exponential function, i.e. the formula $\exp(x,y)=z$ is to be interpreted as \mbox{$x^{y}=z$.}
\end{definition}

An explicit definition of $\exp$ can be written as a $\Delta_{0}$ formula (for more details see Section 3(c) of Chapter V of \cite{hapu98}).

\begin{definition}
The theory $\EA$ is $\IDz$ augmented with the axiom $\forall x\forall y\exists z.\exp(x,y)=z$.
\end{definition}

In the case of our results, the choice of a language and a base theory in principle plays a secondary role, altough the use of $\EA$ allows us to show that the theorems presented do not require a strong base theory --- the axiom $\forall x\forall y\exists z.\exp(x,y)=z$ is $\ISo$-provable (cf. the remark following definition \ref{defi_exp}). An additional advantage of this theory, is that it is finitely axiomatizable (Section 2(c) of Chapter I of \cite{hapu98}). In the following, when we write $\EA$, we mean its finite axiomatisation.

\begin{definition}
Let $\mathcal{L}_{E}$ denotes the language of the signature of $\LA$ augmented with a binary function symbol $E$. The class of $\Delta_{0}(E)$ formulae is the class of $\Delta_{0}$ formulae of $\mathcal{L}_{E}$.
\end{definition}

\begin{definition}
The theory $\IDz(\exp)$ is $\PAm$ augmented with the axioms $\forall x.E(x,0)=1$, $\forall x\forall y.E(x,y+1)=E(x,y)\times x$, and the instances of the parameter induction scheme for $\Delta_{0}(E)$ formulae.
\end{definition}

\begin{fact}
The theory $\EA+\forall x\forall y.exp(x,y)=E(x,y)$ proves each axiom of $\IDz(\exp)$.
\end{fact}

This theorem can be found in the corollary below Theorem 3.3 of \cite{dimgai82}. \cite{pettigrew2008} introduces a set theory which is bi-interpretable with $\EA$ (Theorem 4 and Theorem 7).

\vspace{0.2cm}
\paragraph{The arithmetization of syntax and metalogic}
We explain the basics of arithmetization of syntax in $\EA$ (for more details, consult \cite{hapu98}). Specific issues are introduced when necessary. $\qcr{\varphi}$ denotes the G\"odel code of a formula $\varphi$ and $\num{n}$ denotes the numeral naming $n$, i.e. the expression $0+1+1+\ldots +1$ ($n$ times 1). Usually, we omit quotation marks to increase readability. $x\in\Var$, $x\in\TmcLA$, $x\in\Sent_{\mathcal{L}}$ denote the arithmetical formulae representing the sets of (codes of) variables, closed terms of $\LA$ and (assuming that the coding of $\mathcal{L}$ has been fixed) sentences of $\mathcal{L}$ (respectively). $\val(x)=y$ denotes the arithmetical formula, which expresses "$y$ is the value of the $\LA$-term $x$" and $\dpt(x) = y$ expresses "the depth of a formula $x$ is equal to $y$". The use of definable functions representing syntactical operations is signalised with $[\cdot]$, e.g. "$x=[\alpha\land\beta]$" is a shorthand for "the value of $x$ is equal to the code of a conjunction of sentences $\alpha$ and $\beta$". The addition of a dot above a variable indicates that a numeral is inserted instead of this variable, for example "$y=[\varphi(\dot{x}/v)]$"  is a shorthand for "the value of $y$ is equal to the code of a formula formed by substituting the $x$-th numeral for all occurrences of the variable $v$ into the formula $\varphi(v)$". We avoid mentioning $v$ if it is clear from context. We assume that all operations described above are provably total in $\EA$ and hence allow ourselves to treat them as new function symbols, writing e.g. $\forall x \bigl(x\in\Sent_{\LA}\rightarrow\varphi [\neg x]\bigr)$ to express that the negation of every arithmetical sentence satisfies formula $\varphi(x)$.  If $Th$ is any r.e. theory, then $\Proof_{Th}(y,x)$ denotes the canonical  provability predicate for $Th$ expressing "$y$ is a proof of $x$ from the axioms of $Th$". We put
\begin{align*}
\Prov_{Th}(x)&:= \exists y.\Proof_{Th}(y,x) \\ \Con_{Th}(x)&:= \forall y<x.\neg\Proof_{Th}(y,\qcr{0=1}) \\ \Con_{Th}&:= \forall x.\Con_{Th}(x).
\end{align*} 

\vspace{0.2cm}

We recall the canonical construction of partial truth predicates, $\Tr_n$, such that
\begin{enumerate}
    \item the function $n\mapsto \qcr{\Tr_n(x)}$ is elementary;
    \item for every $n$, $\Tr_n(x)$ satisfies compositional Tarski truth clauses for formulae of logical depth at most $n$.
\end{enumerate}
 The construction of $\Tr_n$ is by elementary induction on $n$: for $n=0$ we put $\Tr_0(x):= \exists s,t\in\TmcLA (x = [s=t]\wedge \val(s) = \val(t)).$ For $n>0$ we put $\Tr_n(x)$ to be the disjunction of the following formulae:
 \begin{itemize}
     \item $\dpt(x)<\num{n} \wedge \Tr_{n-1}(x)$
     \item $\dpt(x) = \num{n}\wedge \exists \varphi (y = [\neg \varphi] \wedge \neg \Tr_{n-1}(\varphi)\bigr)$
     \item $\dpt(x) = \num{n} \wedge \exists \varphi,\psi(x = [\varphi\wedge \psi]\wedge \Tr_{n-1}(\varphi)\wedge \Tr_{n-1}(\psi))$
     \item $\dpt(x) = \num{n}\wedge \exists \varphi\exists v(\Var(v)\wedge x = [\exists v \varphi] \wedge \exists z.\Tr_{n-1}[\varphi(\dot{z}/v)].$
 \end{itemize}

Finally, we recall provable $\Sigma_1$-completeness of $\PAm$:
\begin{fact}
For every $\Sigma_{1}$-sentence $\sigma\in \LA$, $\PA\vdash \sigma\rightarrow\Prov_{\PAm}(\ulcorner\sigma\urcorner)$. In particular, given a theory $Th\subseteq Th(\N)$ and an arithmetical $\Delta_{1}(Th)$-sentence $\tau$,
$\PA+Th\vdash\Prov_{\PAm}(\ulcorner\tau\urcorner)\lor\Prov_{\PAm}(\ulcorner\lnot\tau\urcorner)$.
\end{fact}

\begin{definition}[Flexible formula]
    Let $Th$ be any theory and let $\Gamma$ be a class of formulae. A formula $\theta(x)$ is $\Gamma$-\emph{flexible} over $Th$ iff for every formula $\varphi(x)\in \Gamma$, the theory $Th + \forall x \bigl(\theta(x)\leftrightarrow \varphi(x)\bigr)$ is consistent. 
\end{definition}

\begin{theorem}[Essentially due to Kripke, see \cite{blanck2017}, Theorem 4.2]\label{tw_flex}
For every r.e. theory in the arithmetical language $Th$ extending $\EA$ and every $n$, there exists a $\Sigma_n$-formula, which is $\Sigma_n$-flexible over $Th$.
\end{theorem}
    
The generalisation of this theorem to sequential theories, was proved in \cite{lelyk_wcislo_universal}.


\subsection{Definability}


In this paper we consider only a very special, strong form of the interpretability relation between theories. We focus on theories that extend the language of arithmetic with finitely many relational symbols. In the following, we present the most important relevant definitions (for a detailed presentation see \cite{viss_sivs}):

\begin{definition}
Given three signatures $A$, $B$ and $C\subseteq A\cap B$, a $C$-\emph{conservative translation} from $A$ to $B$ is a mapping $\sigma$ constant on $C$ such that, for every non-arithmetic $n$-ary predicate symbol $P$ of $A$, $\sigma(P)$ is a formula of $B$ with exactly $n$ free variables. Each such translation can naturally be extended to the set of all formulae over the signature $A$. The result of applying the translation $\sigma$ to a sentence $\varphi$ is denoted by $\sigma[\varphi]$.\footnote{As usual, we assume that the translation is preceded with a process of variable renaming (if necessary).}
\end{definition}

\begin{definition}
We say that a theory $U$ is \emph{definable} in a theory $V$  modulo signature $C$ if there exists a translation $\sigma$ from the language of $U$ to the language of $V$ which is $C$-conservative and such that for every axiom $\alpha$ of $U$, $V\vdash \sigma[\alpha]$. 
\end{definition}

\begin{example}
    Consider the theory $U$ extending $\EA$ with an axiom $\forall x.P(x)$, and a theory $V$ extending $\EA$ with an axiom $\forall x.\neg P(x)$, where $P$ is a fresh predicate. Then the translation $\sigma(P) = \neg P(x)$ witnesses that $U$ is definable in $V$ (and actually, $V$ is definable in $U$ via the same translation).
\end{example}

In the rest of the paper we focus exclusively on the translations which are conservative over the arithmetic signature. This is meaningful, since all theories we consider are extensions of $\IDz+\exp$. We note that in the context of the general interpretability theory, our definition of definability corresponds to a direct and $\LA$-conservative interpretability (see \cite{viss_sivs}).

\begin{fact}
The definability relation between theories extending $\EA$ is a pre-order, i.e. it is transitive and reflexive.
\end{fact}


\subsection{Axiomatic Theories of Truth}


We begin by providing a fundamental definition. 

\begin{definition}
A consistent theory $Th$ containing $\EA$ is a \emph{theory of truth} for a language $\mathcal{L}$ iff there exists a formula $\Theta$ with one free variable such that for each sentence $\sigma\in\mathcal{L}$
$$Th\vdash\Theta(\ulcorner\sigma\urcorner)\leftrightarrow\sigma\text{.}$$
\end{definition}

\begin{theorem}[Tarski]
Suppose that $Th$ is a consistent theory extending $\EA$ in a language $\mathcal{L}$. Then $Th$ is not a theory of truth for $\mathcal{L}$.
\end{theorem}

\begin{definition}[Some prominent theories of truth for $\LA$] The first three theories are formulated in the language $\LT:=\LA\cup \{T\}$, where $T$ is a fresh unary predicate. The last theory is formulated in the language $\LA\cup \{D\}$, where $D$ is a fresh binary predicate.

$\TBm$ extends $\EA$ with infinitely many axioms of the form
\[T(\qcr{\sigma})\leftrightarrow \sigma,\]
where $\sigma$ is an $\LA$-sentence.

$\UTBm$ extends $\EA$ with infinitely many axioms of the form
\[\forall x \bigl(T[\sigma(\dot{x})]\leftrightarrow \sigma(x)\bigr),\]
where $\sigma(x)$ is an $\LA$-formula with at most $x$ free.

$\CTm$ extends $\EA$ with the following axioms
\begin{itemize}
    \item $\forall s,t \big(s,t\in\TmcLA\rightarrow\big(T[s=t]\leftrightarrow\val(s)=\val(t)\big)\big)$,
    \item $\forall\varphi\big(\varphi\in\Sent_{\LA}\rightarrow\big(T[\lnot\varphi]\leftrightarrow\lnot T(\varphi)\big)\big)$,
    \item $\forall\varphi,\psi\big(\varphi,\psi\in\Sent_{\LA}\rightarrow(T[\varphi\land\psi]\leftrightarrow T(\varphi)\land T(\psi)\big)\big)$,
    \item $\forall\varphi\forall y\big(y\in\Var\land[\exists y.\varphi]\in\Sent_{\LA}\rightarrow\big(T[\exists y.\varphi]\leftrightarrow\exists z.T[\varphi(\dot{z}/y)]\big)\big)$.
\end{itemize}

$\DEFm$ extends $\EA$ with all axioms of the form
\[\forall z \bigl(D(\qcr{\sigma},a)\leftrightarrow \exists ! x. \sigma\wedge \sigma(a)\bigr),\]
where $\sigma(x)$ is an $\LA$-formula with at most $x$ free and $\exists ! x$ is understood as "There is exactly one $x$". The intuitive reading of $D(\sigma,x)$ is "$x$ is definable by the formula $\sigma$."
\end{definition}

\begin{proposition}
$\TBm$, $\UTBm$, $\CTm$ and $\DEFm$ are axiomatic theories of truth.
\end{proposition}
\begin{proof}
This is immediate for $\TBm, \UTBm$ and $\CTm$. For $\DEFm$ it is enough to consider the formula
\[\Theta(\sigma):= D([\sigma\wedge x=0], 0).\]
The intuition is that for every $\LA$-sentence $\sigma$, provably in $\DEFm$, $\Theta(\qcr{\sigma})$ holds iff the sentence $\sigma\wedge x=0$ is a definition of $0$, which is equivalent to $\sigma$ being true.
\end{proof}

\begin{definition}
$\TB$, $\UTB$, $\CT$, $\DEF$ denote the extensions of theories $\TBm$, $\UTBm$, $\CTm$ and $\DEFm$ with full induction scheme for their languages.
\end{definition}

\begin{definition}\label{def_ct0}
$\CTz$ extends $\CTm$ with induction axioms for $\Delta_0$-formulae of $\LT$.
\end{definition}

The following fact is a crucial model-theoretical property of $\CTz$. For a proof see \cite[Fact 31]{lelyk_wcislo_local}.

\begin{fact}\label{fact_modeltheory_ct_0}
Let $(\mathcal{M},T)\models \CTm$. For $c\in M$, put $T_c:= \{\varphi\in M:\ (\mathcal{M},T)\models \dpt(\varphi)\leqslant c \wedge T(\varphi)\}$. Then the following are equivalent:
\begin{itemize}
    \item $(\mathcal{M}, T)\models \CTz$.
    \item For every $c$, $(\mathcal{M}, T_c)\models \CTm + \IND(\LT).$
\end{itemize}
In particular, $\CTz\vdash \PA$.
\end{fact}

\begin{definition}
    Let $A\subseteq B$ be two theories (possibly in different languages). We say that $B$ is \emph{conservative} over $A$ if every sentence $\varphi$ in the language of $A$, which is provable in $B$, is provable in $A$. We say that $B$ is semantically (model-theoretically) conservative over $A$ if every model of $A$ can be expanded to a model of $B$.\footnote{We say that a model $\mathcal{M}'$ is an expansion of a model $\mathcal{M}$, if $\mathcal{M}$ can be obtained from $\mathcal{M}'$ by simply erasing the interpretations of some symbols.}
\end{definition}

\begin{definition}[Definition of truth, the definability order]
    A theory of truth $Th$ for $\mathcal{L}$ is called a \emph{definition of truth} iff $Th$ is finitely axiomatizable. The set of definitions of truth is denoted by $\DEF$. The set of definitions of truth which are conservative extensions of $\EA$ is denoted by $\cDEF$. The \emph{definability order} is the structure $\langle \cDEF, \LHD\rangle.$\footnote{We tacitly assume that this is in fact a partial order, by passing to the quotient set of $\cDEF$ modulo the equivalence relation of mutual definability.}
\end{definition}

\begin{remark}\label{rem_cons_init_def}
    $\langle \cDEF, \LHD\rangle$ forms an initial segment of $\langle \DEF,\LHD\rangle,$ that is if $\beta\LHD\alpha$ and $\alpha\in \cDEF$, then $\beta\in \cDEF$ as well, and for every $\alpha\in\DEF$, there is a $\beta\in \cDEF$ such that $\beta\LHD\alpha$.
\end{remark}

\begin{fact}\label{fakt_definicje_prawdy}
Among the aforementioned theories of truth, only $\CTm$ and $\CTz$ are definitions of truth. For $\CTm$ this is clear from the presentation. As shown e.g. in \cite{lelyk_global} over $\CTm$, $\CTz$ is equivalent to a single sentence $\forall \varphi \bigl(\Prov_{\PA}(\varphi)\rightarrow T(\varphi)\bigr).$ $\CT^-$ is a conservative extension of $\EA$, but $\CTz$ is not, as its arithmetical consequences properly exceeds those of $\PA$.
\end{fact}


\subsubsection{Between Disquotation and Compositionality}\label{sec_between_disq_and_comp}


\begin{definition}\label{def_ctm(x)}
$\CTm(x)$ is the conjunction of the following formulae of $\LT$:

\begin{itemize}
    \item $\EA$
    \item $\forall s,t \big(s,t\in\TmcLA\rightarrow\big(T[s=t]\leftrightarrow\val(s)=\val(t)\big)\big)$,
    \item $\forall\varphi\big(\sigma\in\Sent_{\LA}\land\dpt[\lnot\varphi]\leqslant x\rightarrow\big(T[\lnot\varphi]\leftrightarrow\lnot T(\varphi)\big)\big)$,
    \item $\forall\varphi,\psi\big(\varphi,\psi\in\Sent_{\LA}\land\dpt[\varphi\land\psi]\leqslant x\rightarrow(T[\varphi\land\psi]\leftrightarrow T(\varphi)\land T(\psi)\big)\big)$,
    \item $\forall\varphi\forall v\big(v\in\Var\land[\exists v.\varphi]\in\Sent_{\LA}\land\dpt[\exists v.\varphi]\leqslant x\rightarrow\big(T[\exists v.\varphi]\leftrightarrow\exists z.T[\varphi(\dot{z}/v)]\big)\big)$.
\end{itemize}
\end{definition}

\begin{proposition}
For every $n$, $\EA\RHD\CTm(\num{n})$.
\end{proposition}
\begin{proof}[Sketch of proof]
This follows since for every $n$ there is an arithmetical formula $\Tr_{n}(x)$, which, provably in $\EA$ satisfies Tarski's inductive truth clauses for formulae of complexity at most $x$. 
\end{proof}

The proposition below was well known for some time already (a variant of it is given e.g. in \cite{bekpakhannals} and \cite{viss19enayat}). We include a simple proof for the Reader's convenience.

\begin{proposition}
$\{\CTm(\num{n}):\ n\in\omega\}\RLHD\UTBm$
\end{proposition}
\begin{proof}[Sketch of proof]
By a standard argument using induction on the depth of formulae, it is easy to show that actually $\{\CTm(\num{n}):\ n\in\omega\}\vdash \UTBm$. In the reverse direction, we define the following truth predicate $T'(x)$:
\[\Sent_{\LA}(x)\wedge \exists z (\dpt(x) = z\wedge T[\Tr_z(\dot{x})]).\]
In the above $\Tr_z$ denotes the representation of the elementary mapping $n\mapsto \qcr{\Tr_n(x)}$ described in Section \ref{sec_arithm}. One easily checks that the mapping $\sigma$ given by $\sigma(T)= T'(x)$ makes each $\sigma(\CT^-(\num{n}))$ provable in $\UTBm$. 

\end{proof}

The above proposition witnesses that $\UTBm$ can be treated as the infimum of compositional theories: it can accomodate a truth predicate which satisfies compositional Tarski truth conditions for all formulae of standard complexity. Therefore it is natural to treat a definition of truth as compositional if it can define the theory $\UTBm$. This is the philosophical intuition underlying the following formal question by Albert Visser (see \cite{viss19enayat}):

\begin{question}[Visser]
    Is there a definition of truth (for $\LA$) which does not define $\UTBm$?
\end{question}

This is the question which initiated our research. Unfortunately it still remains unanswered.


\subsection{Algebraical Preliminaries}\label{sect_algebra}


We take all the basic definitions from \cite{gratz_lattice} and \cite{Handbook_BAs}.

\begin{definition}[Lattice]
A \emph{lattice} is a set $L$ with two operations $\cap$ and $\cup$ each of which is associative and symmetric and additionally satisfying $a\cap(b\cup a) = a$ and $a\cup(b\cap a) = a$, for all $a,b\in L$. A lattice $L$ is \emph{distributive} if for all $a,b,c$ it holds that $a\cup(b\cap c) = (a\cup b)\cap (a\cup c)$ and $a\cap(b\cup c) = (a\cap b)\cup (a\cap c)$. 
\end{definition}

One can show that each lattice $L$ originates from a partial order $(L, \leqslant_{L})$ by putting $a\cup b = \sup\{a,b\}$ and $a\cap b = \inf\{a,b\}$. In fact, this partial order $\leqslant_L$ is canonically determined by $L$ by letting $a\leqslant_L b := a\cap b = a$. Conversely, for every partial order, so defined structure is a lattice.

\begin{definition}[Countable universal distributive lattice]
    Lattice $L$ is \emph{countable universal} iff $L$ contains an isomorphic copy of every countable lattice.
\end{definition}

\begin{definition}[Boolean algebra]
A \emph{Boolean algebra} is a structure $\langle B, \cap, \cup, \mathbb{0}, \mathbb{1}, \neg\rangle$ such that $\neg$ is a unary function, $\mathbb{0}, \mathbb{1}\in L$ and $\langle B, \cap, \cup\rangle$ is a distributive lattice,  which additionally satisfies $\forall a\in L (a\cup\mathbb{1} = \mathbb{1} \wedge \mathbb{0}\cap a = \mathbb{0})$ and $\forall a \in L(a\cap \neg a = \mathbb{0} \wedge a \cup \neg a = \mathbb{1}).$

An atom in a Boolean algebra is any nonzero element $a$, such that for every $b$, if $b\leqslant_B a$, then $b = a\vee b = \mathbb{0}$. A Boolean algebra $B$ is called \emph{atomless} iff there are no atoms in $B$.
\end{definition}

It is an easy exercise in back-and-forth arguments that any two countable atomless Boolean algebras are isomorphic. Additionally, we have the following observation which we use in the rest of the paper.
\begin{fact}
    If $\langle B, \cap,\cup, \mathbb{0},\mathbb{1}, \neg\rangle$ is the countable atomless Boolean algebra, then $\langle B, \cap, \cup\rangle$ is a countable universal distributive lattice. 
\end{fact}

For a proof see Theorem 10 of \cite{vermeer}.

\begin{example}[An elementarily presented countable atomless Boolean algebra]
    There are many different ways of realizing the countable atomless Boolean algebra. In what follows we need one, whose basic properties are provable in $\EA$. One of the simplest such algebras is the Boolean algebra of (equivalence classes) of propositional formulae over countable many propositional variables. To be more precise, let Prop be the set of all formulae of propositional logic over $\{p_i\}_{i\in\omega}.$ Let $B$ be the set of equivalence classes of Prop modulo propositional equivalence, $\sim$. We define the operations of $\cap$, $\cup$, $\neg$ and $\mathbb{0}$ and $\mathbb{1}$ in the canonical way, for example
    \[[\varphi]_{\sim}\cap [\psi]_{\sim} = [\varphi\wedge \psi]_{\sim}.\]
    Let us observe that there is the canonical way to represent $B$ inside $\EA$: since $\sim$ is an elementary relation, we can take the elements of $B$ to be the least (according to a fixed coding of Prop) elements of equivalence classes. In this way, we obtain elementary formulae $\delta_B(x)$, $\cap(x,y,z)$, $\cup(x,y,z)$, $\mathbb{0}(x)$, $\mathbb{1}(x)$, $\neg(x,y)$ such that
    \[\EA \vdash "\langle \delta_B(x), \cap(x,y,z), \cup(x,y,z), \mathbb{0}(x), \mathbb{1}(x), \neg(x,y)\rangle \textnormal{ is an atomless Boolean algebra }".\]
\end{example}

\begin{definition}
A subset $\mathcal{F}$ of a Boolean algebra $\mathbb{B}$ is a \emph{filter} if and only if it satisfies the following:
\begin{itemize}
    \item $\mathbb{1}\in\mathcal{F}$,
    \item $\forall x, y\in\mathbb{B}(x\in\mathcal{F}\land x\leqslant_{\mathbb{B}}y\rightarrow y\in\mathcal{F})$,
    \item $\forall x, y\in\mathbb{B}(x\in\mathcal{F}\land y\in\mathcal{F}\rightarrow x\cap y\in\mathcal{F})$.
\end{itemize}
It is an \emph{ultrafilter} if additionally $\mathbb{0}\notin\mathcal{F}$ and every filter $\mathcal{E}$ properly containing $\mathcal{F}$ is the entire algebra. 
\end{definition}

Ultrafilters are sometimes described as measures of the size of elements of a Boolean algebra (elements of an ultrafilter are considered large). We also use a different point of view.

\begin{fact}\label{fact_ultrafilters}
Let $\mathcal{F}$ be a subset of a Boolean algebra $\mathbb{B}$. The following are equivalent:
\begin{itemize}
    \item $\mathcal{F}$ is an ultrafilter;
    \item $\mathcal{F}$ is a filter, $\forall x(x\in\mathcal{F}\lor\lnot x\in\mathcal{F}$), and $\mathbb{0}\notin\mathcal{F}$;
    \item $\mathcal{F}$ satisfies
    \begin{itemize}
        \item $\mathbb{1}\in\mathcal{F}$ and $\mathbb{0}\notin\mathcal{F}$,
        \item $\forall x,y(x\in\mathcal{F}\land y\in\mathcal{F}\leftrightarrow x\cap y\in\mathcal{F})$,
        \item $\forall x,y(x\in\mathcal{F}\lor y\in\mathcal{F}\leftrightarrow x\cup y\in\mathcal{F})$,
        \item $\forall x(x\in\mathcal{F}\lor\lnot x\in\mathcal{F})$.
    \end{itemize}
\end{itemize}
\end{fact}

For more informations see Section 2.2. of \cite{Handbook_BAs}.



\section{Structure of Definability Order}


\subsection{Dense Subsets of $\DEF$}


We begin by presenting a tiny generalisation of the argument due to Pakhomov and Visser (see \cite{viss19enayat} and \cite{pakhomov_visser}). The proof uses a class of natural weakly compositional definitions of truth. The definition of this class utilises the formula $\CTm(x)$ as introduced in Definition \ref{def_ctm(x)}.

\begin{definition}[Weakly compositional definitions of truth]\label{wCDEF}
For a formula $\varphi(x)$, we define $\CTo_{\varphi(x)}$ as $\EA\ +\ \forall x(\varphi(x)\rightarrow\CTm(x))$.
\end{definition}

\begin{theorem}[Essentially Pakhomov-Visser]\label{tw_pakh-viss}
$\DEF$ is not $\Sigma_2$-definable.
\end{theorem}
\begin{proof}
Aiming at a contradiction, suppose the contrary. Let $\sigma(x)$ be a $\Sigma_2$-definition of $\DEF$. By the fixed point lemma, let $\delta(x)$ be a $\Pi_1$-formula such that $\EA$ proves 
$$\exists x.\delta(x)\leftrightarrow\sigma(\qcr{\CTo_{\forall y<x.\lnot\delta(y)}}){.}$$
More precisely, the fixed point lemma is applied to the following formula with a free variable $z$:
\begin{equation}\label{equat_fxp}
    z\in\Sent_{\LA}\land\exists \delta\in\Pi_1\bigl(z=[\exists x.\delta] \land\sigma[\CTo_{\forall y<x.\neg\delta(y)}]\bigr)\text{.}
\end{equation}
In the above, $z=\exists x.\delta$ denotes a formula with two variables $z$ and $\delta$, formalising a relation "$z$~is the G\"odel code of a sentence formed by prefixing a formula $\delta$ with an existential quantifier binding the variable $x$". Similarly, $\sigma[\CTo_{\forall y<x.\neg\delta(y)}]$ is a formula with a free variable~$\delta$. When constructing a fixed point formula, we make sure that it is in the proper $\Sigma_2$-form, beginning with exactly one existential quantifier. This is possible by the standard textbook proof of the diagonal lemma (see e.g. \cite{hapu98}, Chapter III, Theorem 2.1).\footnote{We stress that we are not taking an arbitrary fixed point of \eqref{equat_fxp}. It is easy to observe that, for trivial reasons, any disprovable formula which is not in a proper $\Sigma_2$-form will be a fixed point of \eqref{equat_fxp}.}

\medskip
The sentence $\exists x.\delta(x)$ must then be false in the standard model. To prove this, suppose the opposite. Then $\CTo_{\forall y< x.\lnot\delta(y)}\in \DEF$. At the same time there exists the smallest natural number $n$ such that $\N\models\delta(\num{n})$. It follows that
$$\N\models\forall y<x.\lnot\delta(y)\leftrightarrow x<\num{n}\text{.}$$ Hence $\CTo_{\forall y< x.\lnot\delta(y)} + \Th(\N)$ is equivalent to $\CTm(\num{n})+\Th(\N)$, where $\Th(\N)$ denotes the set of arithmetical sentences which are true in the standard model of arithmetic. It follows that $\CTm(\num{n})+\Th(\N)$ is a theory of truth for $\LA$. But $\Th(\N)\RHD\CTm(\num{n})+\Th(\N)$, using partial truth predicates $\Tr_n$. Hence $\Th(\N)$ is a theory of truth for $\LA$. This contradicts Tarski's theorem on the undefinability of truth.

\medskip
Since the sentence $\exists x.\delta(x)$ is false in the standard model, $\CTo_{\forall y<x.\lnot\delta(y)}\notin \DEF$. At the same time $\EA\vdash\lnot\delta(\num{n})$ for every natural number $n$, by $\Sigma_{1}$-completeness. Hence $\CTo_{\forall y<x.\lnot\delta(y)}$ is a definition of truth. A contradiction.

\end{proof}

\medskip

We say that a subset $B$ of a pre-order $(A,\leqslant)$ is \emph{dense} if for every $a\in A$ there is $b\in B$ such that $b\leqslant a$.
\begin{corollary}
    Suppose that $B$ is a $\Sigma_2$-definable subset of $\DEF$. Then $B$ is not dense in $\langle\DEF,\LHD\rangle.$
\end{corollary}
\begin{proof}
    Assume that $B$ is dense in $\langle\DEF, \LHD\rangle$, then  $\sigma(x):=\exists \beta\in B. \beta\LHD x$ is a definition of $\DEF$. If $B$ had a $\Sigma_2$ definition, then $\sigma(x)$ would be $\Sigma_2$, contradicting Theorem \ref{tw_pakh-viss}.
\end{proof}

\begin{corollary}
There are no minimal elements in $\langle \DEF, \LHD\rangle$.
\end{corollary}
\begin{proof}
    Fix any $\alpha\in\DEF$. Since finite subsets are $\Sigma_2$-definable, then $\{\alpha\}$ is not dense, so there is $\beta$ such that $\beta\not\RHD\alpha$. Then $\alpha\vee\beta$ is a definition of truth which is strictly smaller than $\alpha$ in $\langle\DEF,\LHD\rangle.$ Indeed, let $\Theta_{\alpha}(x)$ and $\Theta_{\beta}(x)$ be two formulae interpreting the truth predicate in $\alpha$ and $\beta$ respectively. Then 
    $$(\alpha\rightarrow\Theta_{\alpha}(x))\land(\lnot\alpha\rightarrow \Theta_{\beta}(x))$$
acts as a truth-predicate in $\alpha\lor\beta$. Finally, $\alpha\RHD\alpha\lor\beta$, because $\alpha\vdash\alpha\lor\beta$, and $\alpha\lor\beta\not\RHD\alpha$, because otherwise $\beta\RHD\alpha$, as $\beta\RHD\alpha\lor\beta$.

\end{proof}

By the remark immediately following Fact \ref{fakt_definicje_prawdy} we obtain the following
\begin{corollary}
    There are no minimal elements in $\langle \cDEF, \LHD\rangle.$
\end{corollary}


\subsection{Countable Universal Distributive Lattice}

\subsubsection{Suprema and Infima}


\begin{restatable}{theorem}{thmsupinf}
The definability order forms a distributive lattice. The infimum and supremum operations are derived from the propositional operations as follows.
\begin{itemize}
    \item The infimum of $\alpha$ and $\beta$ is $\alpha\lor\beta$.
    \item The supremum of $\alpha$ and $\beta$ is $\alpha\land\repsymb_{\alpha}(\beta)$, where $\repsymb_{\alpha}(\cdot)$ is a fixed translation that replaces non-arithmetic relational symbols in $\cdot$ with symbols of the same arity that are not present in~$\alpha$.
\end{itemize}
\end{restatable}
The proof is sketched below. See Appendix B for the complete proof.
\begin{proof}[Sketch of proof]
The proof has three parts. First, we show that $\alpha\lor\beta$ and $\alpha\land\repsymb_{\alpha}(\beta)$ are syntactically conservative over $\EA$ definitions of truth (whenever $\alpha$ and $\beta$ are such). Then we show that $\alpha\lor\beta$ is an infimum and $\alpha\land\repsymb_{\alpha}(\beta)$ is a supremum of $\alpha$ and $\beta$ (the proof also shows that the operations are well defined as operations on the equivalence classes of truth definitions modulo the relation of mutual definability, i.e. if $\alpha\RLHD\alpha'$ and $\beta\RLHD\beta'$, then $\alpha\lor\beta\RLHD\alpha'\lor\beta'$ and $\alpha\land\repsymb_{\alpha}(\beta)\RLHD\alpha'\lor\repsymb_{\alpha'}(\beta')$). We complete the proof by showing that the structure satisfies distributivity axioms. Only the first part of the proof is presented in full here.

\medskip
Both $\alpha\lor\beta$ and $\alpha\land\repsymb_{\alpha}(\beta)$ extend $\EA$. Given a sentence $\sigma\in\LA$, suppose that $\alpha\lor\beta\vdash\sigma$, then in particular $\alpha\vdash\sigma$, so $\EA\vdash\sigma$ by the conservativity of $\alpha$. Now suppose that $\alpha\land\repsymb_{\alpha}(\beta)\vdash\sigma$, then
$$\emptyset\vdash\alpha\rightarrow(\repsymb_{\alpha}(\beta)\rightarrow\sigma)\text{.}$$
Since the sets of the symbols present in $\alpha$ and $\repsymb_{\alpha}(\beta)$ respectively, intersect exactly at the arithmetic symbols, by Craig's interpolation theorem (see 2.5.5 of \cite{Handbook_Prf_Th}), there exists an arithmetical sentence $\tau$ such that
$$\emptyset\vdash\alpha\rightarrow\tau\text{ \ \ and \ \ }\emptyset\vdash\tau\rightarrow(\repsymb_{\alpha}(\beta)\rightarrow\sigma)\text{,}$$
and hence $\emptyset\vdash\repsymb_{\alpha}(\beta)\rightarrow(\tau\rightarrow\sigma)$. The arithmetical consequences of $\alpha$ and $\repsymb_{\alpha}(\beta)$ are exactly the same, so $\emptyset\vdash\repsymb_{\alpha}(\beta)\rightarrow\tau$, and therefore $\repsymb_{\alpha}(\beta)\vdash\sigma$. Hence $\EA\vdash\sigma$, since $\repsymb_{\alpha}(\beta)$ is syntactically conservative over $\EA$.

\medskip
The second part of the proof is analogous to the one given in \cite{mycielski}. The "$x\cap(y\cup z)=(x\cap y)\cup(x\cap z)$" part of the third part of the proof is immediate: given definitions of truth $\alpha$, $\beta$ and $\gamma$,
$$\alpha\land\repsymb_{\alpha}(\beta\lor\gamma)=(\alpha\land\repsymb_{\alpha}(\beta))\lor(\alpha\land\repsymb_{\alpha}(\gamma))\text{.}$$
The "$x\cup(y\cap z)=(x\cup y)\cap(x\cup z)$" part is a little less straightforward, as the supremum operation involves a language change. To deal with this, we show how to build suitable "check-case-by-case" translations.
\end{proof}



\subsubsection{Main Result}\label{sec_main}


In this section, we prove that the distributive lattice determined by the definability order between definitions of truth is a countable universal distributive lattice. 

\medskip
The range of our embedding consists of weakly compositional defintions of truth -- the definition below uses truth definitions of the form $\CTo_{\varphi}$ as introduced in Definition \ref{wCDEF}.

\begin{definition}[Codomain] 
For a formula $\varphi(x)$, we define
$$\Delta:=\{\CTo_{\varphi(x)}:\ \forall n\in\omega\ \EA\vdash\varphi(\num{n})\ \land \ \EA\vdash\forall x(\varphi(x)\rightarrow\forall y<x.\varphi(y))\}.$$
\end{definition}

\begin{remark}
Each member of $\Delta$ is a syntactically conservative over $\EA$ definition of truth. Moreover, if $\CTo_{\varphi(x)}\in\Delta$, then, provably in $\EA$, $\varphi(x)$ defines an initial segment containing all natural numbers.
\end{remark}

The two lemmata below are crucial in arguing for comparability and incomparability of elements of $\Delta$ with respect to $\LHD$.

\begin{lemma}[Comparability lemma]\label{lem_comp}
For $\CTo_{\alpha}$ and $\CTo_{\beta}$ being members of $\Delta$,
$$\text{if }\EA\vdash\forall x(\alpha(x)\rightarrow\beta(x))\text{, then }\CTo_{\alpha}\LHD\CTo_{\beta}\text{.}$$ 
\end{lemma}

\begin{remark}
The antecedant can be weakened to $\exists n\in\omega\ \EA\vdash\forall x(\alpha(x+\num{n})\rightarrow\beta(x))$.
\end{remark}

\begin{proof}
If $\EA\vdash\forall x(\alpha(x)\rightarrow\beta(x))$, then
$$\EA\vdash\forall x(\beta(x)\rightarrow\CTm(x))\rightarrow\forall x(\alpha(x)\rightarrow\CTm(x))\text{.}$$
Hence $\CTo_{\beta}\vdash\CTo_{\alpha}$. Thus $\CTo_{\beta}\RHD\CTo_{\alpha}$ via the identity translation.
\end{proof}
\begin{remark}
To prove the remark, we can use a partial truth predicate built on the top of the truth predicate of~$\CTo_{\beta}$ (we omit the details).
\end{remark}

\begin{lemma}[Incomparability lemma]\label{order2lemma}
For $\CTo_{\alpha}$ and $\CTo_{\beta}$ being members of $\Delta$,
$$\text{if }\CTo_{\alpha}\LHD\CTo_{\beta}\text{, then }\CTz\vdash\forall x.\alpha(x)\rightarrow\forall x.\beta(x)\text{.}$$
\end{lemma}
\begin{proof}
For a contradiction, suppose that $\CTo_{\beta}\RHD\CTo_{\alpha}$ and that $\CTz+\forall x.\alpha(x)\land\exists x.\lnot\beta(x)$ is consistent. In particular, by the second G\"odel's incompleteness theorem, there is a model 
$$\langle\mathcal{M},T\rangle\models\CTz+\forall x.\alpha(x)\land\exists x.\lnot\beta(x)+\lnot\Con_{\CTz+\forall x.\alpha(x)\land\exists x.\lnot\beta(x)}\text{.}$$
Let $c\in\mathcal{M}$ be the smallest element such that $\mathcal{M}\models\lnot\beta(c)$ -- such an element exists since $\CTz\vdash\PA$ and $\beta$ is an arithmetical formula. Let $T_{c}$ be the set $\{\varphi\in\mathcal{M}:\ \langle\mathcal{M},T\rangle\models T(\varphi)\land\dpt(\varphi)\leqslant c\}$. Then $\langle\mathcal{M},T_{c}\rangle\models\CTo_{\beta}$ and simultaneously $\langle\mathcal{M},T_{c}\rangle$ satisfies the induction scheme for formulae of $\LT$ by Fact \ref{fact_modeltheory_ct_0}.

\medskip
We assumed that $\CTo_{\beta}\RHD\CTo_{\alpha}$. This means that there exists $\langle\mathcal{M},T_{c}\rangle$-definable set $T'$ such that $\langle\mathcal{M},T'\rangle\models\CTo_{\alpha}$. Since $T'$ is definable in $\langle \mathcal{M}, T_c\rangle$ and $\langle\mathcal{M},T_{c}\rangle$ satisfies the induction axioms for all formulae of $\LT$, then also $\langle \mathcal{M}, T'\rangle$ satisfies all induction axioms for formulae of $\LT$. Hence $\langle\mathcal{M},T'\rangle\models\CT$, because $\mathcal{M}\models\forall x.\alpha(x)$.

\medskip
For a natural number $n$, let $\Tr^{\LT}_n$ denote the partial truth predicate for $\LT$-sentences of logical depth at most $n$, for which $\CT$ proves the Tarski compositional clauses. Let us observe that the cut elemination theorem is provable in Peano Arithmetic (see \cite{hapu98}, ch. V, sec. 5(c), thm~5.17), so, by induction on the lengths of proofs, for each natural number $n$ and $\LT$-sentences $\sigma$ and $\tau$ of syntactical depth not greater than $n$, $\CT$ proves $\Tr^{\LT}_{n}(\ulcorner\sigma\urcorner)\land\Prov_{\sigma}(\ulcorner\tau\urcorner)\rightarrow\Tr^{\LT}_{n}(\ulcorner\tau\urcorner)$. In particular, $\CT$ proves
$$\Tr^{\LT}_{n}(\ulcorner\CTz+\forall x.\alpha(x)\land\exists x.\lnot\beta(x)\urcorner)\land\text{Pr}_{\CTz+\forall x.\alpha(x)\land\exists x.\lnot\beta(x)}(\qcr{0=1})\rightarrow\Tr^{\LT}_{n}(\qcr{0=1})\text{.}$$

\medskip
Because $\langle\mathcal{M},T'\rangle$ is a model of $\CTz+\forall x.\alpha(x)\land\exists x.\lnot\beta(x)$ and simultaneously it satisfies $\CT$, it follows from the above considerations that $\langle\mathcal{M},T'\rangle\models\Con_{\CTz+\forall x.\alpha(x)\land\exists x.\lnot\beta(x)}$. But this means that $\mathcal{M}\models\Con_{\CTz+\forall x.\alpha(x)\land\exists x.\lnot\beta(x)}$ and therefore $\langle\mathcal{M},T\rangle\models\Con_{\CTz+\forall x.\alpha(x)\land\exists x.\lnot\beta(x)}$. This contradicts our initial assumption.
\end{proof}

\begin{remark}
    Ideally, we would like to collapse the comparability and incomparability lemmata into a single lemma characterising the definability relation between theories of the form $\CTo_{\alpha}$ in terms of provability in a fixed theory. At this point we see no natural way of doing this. Various obstacles and possible ways to overcome them are discussed in Appendix A.\footnote{We thank the anonymous referee for suggestions that led to this discussion.}
\end{remark}

\begin{lemma}[Infima \& Suprema]\label{infsuplemma}
For $\CTo_{\alpha}$ and $\CTo_{\beta}$ being members of $\Delta$, $\CTo_{\alpha\land\beta}$ and $\CTo_{\alpha\lor\beta}$ are members of $\Delta$. In the order generated by $\langle\Delta,\LHD\rangle$,
$$\inf\{\CTo_{\alpha},\CTo_{\beta}\}=\CTo_{\alpha\land\beta}\ \ \text{and}\ \ \sup\{\CTo_{\alpha},\CTo_{\beta}\}=\CTo_{\alpha\lor\beta}\text{.}$$
Moreover, these operations are the same in the definability order $\langle \cDEF,\LHD\rangle$.
\end{lemma}
\begin{proof}
Both $\alpha$ and $\beta$ are members of $\Delta$, so they determine cuts in any model of $\EA$. Hence $\EA\vdash\big(\forall x(\alpha(x)\rightarrow\CTm(x))\lor\forall x(\beta(x)\rightarrow\CTm(x))\big)\leftrightarrow\forall x(\alpha(x)\land\beta(x)\rightarrow\CTm(x))$.
and $\EA\vdash\big(\forall x(\alpha(x)\rightarrow\CTm(x))\land\forall x(\beta(x)\rightarrow\CTm(x))\big)\leftrightarrow\forall x(\alpha(x)\lor\beta(x)\rightarrow\CTm(x))$.
In~consequence, $\CTo_{\alpha}\land\CTo_{\beta}\leftrightarrow\CTo_{\alpha\lor\beta}$ and $\CTo_{\alpha}\lor\CTo_{\beta}\leftrightarrow\CTo_{\alpha\land\beta}$. To prove the moreover part, it is enough to observe that $\CTo_{\alpha}\wedge \repsymb_{\CTo\alpha}(\CTo_{\beta})\RLHD\CTo_{\alpha}\land\CTo_{\beta}$. The definability $\CTo_{\alpha}\land\CTo_{\beta}\RHD \CTo_{\alpha}\wedge \repsymb_{\CTo\alpha}(\CTo_{\beta})$ is straightforward, and the converse direction is witnessed by a disjunctive translation which first measures which one of $\alpha,\beta$ defines longer initial segment and then chooses the right truth predicate for the interpretation. Formally, the interpretation of the predicate is the following formula:
\[\bigl(\alpha\subseteq \beta \rightarrow \repsymb_{\CTo\alpha}(T)(x)\bigr)\land \bigl(\beta\subsetneq \alpha\rightarrow T(x)\bigr).\]
In the above, $\alpha\subseteq\beta$ abbreviates $\forall x(\alpha(x)\rightarrow \beta(x))$ and the abbreviation $\beta\subsetneq \alpha$ is unravelled accordingly.
\end{proof}

\begin{definition}[Domain of the embedding]
We denote by $\mathbb{B}=\langle B,\mathbb{0},\mathbb{1},\cap,\cup,\lnot\rangle$ a fixed countable atomless Boolean algebra such that its domain and all symbols can be represented in $\N$ by elementary formulae such that $\EA$ proves that $\mathbb{B}$ is the countable atomless Boolean algebra. We denote by the above symbols also the formulae representing~the algebra. We let $\leqslant_{\mathbb{B}}$ be the partial order determined by $\cap$.  An example of such a Boolean algebra is given in Section \ref{sect_algebra}.
\end{definition}

\begin{definition}[Embedding]
We define a function $F:B\longrightarrow\Delta$. Let $\eta(y)$ be a $\Sigma_{1}$-formula \mbox{$\Sigma_{1}$-flexible} over arithmetical consequences of $\CTz+\neg\Con_{\CTz}$ (see Theorem \ref{tw_flex}).  For $a\in B$, we define $\varphi_{a}(x)$ as $\Con_{\CTz}(x)\ \lor\ $"$\{y:\ \eta(y)\}$ is an ultrafilter on~$\mathbb{B}$ containing~$\num{a}$" and put $F(a)=\CTo_{\varphi_{a}}$. By "$\{y:\ \eta(y)\}$ is an ultrafilter on~$\mathbb{B}$ containing~$\num{a}$" we mean the conjunction of the following sentences (cf. \ref{fact_ultrafilters}):
\begin{itemize}
    \item $\forall y(\eta(y)\rightarrow y\in B)$,
    \item $\eta(\mathbb{1})\land\lnot\eta(\mathbb{0})\land\eta(\num{a})$,
    \item $\forall y,z\in B(\eta(y)\land\eta(z)\rightarrow\eta(y\cap z))$,
    \item $\forall y,z\in B(\eta(y)\land y\leqslant_{\mathbb{B}} z\rightarrow\eta(z))$,
    \item $\forall y\in B(\eta(y)\lor\eta(\lnot y))$.
\end{itemize}
\end{definition}

\begin{remark}
Since for each $n\in\omega$, $\Con_{\CT_0}(\num{n})$ is a true $\Delta_0$-sentence, then for each $a\in B$ and  $n\in\omega$, $\EA\vdash\varphi_{a}(\num{n})$. Downward closure of each $\varphi_a(x)$ is immediate. In~consequence for each $a\in B$, $\CTo_{\varphi_{a}}$ is a member of $\Delta$.
\end{remark}

\begin{theorem}[Main Theorem]\label{tw_main}
The above defined function $F$ is a lattice embedding. Therefore the definability order $\langle\cDEF,\LHD\rangle$ is a~countable universal distributive lattice.
\end{theorem}
\begin{proof}
The second part of the theorem follows by the first part and Lemma \ref{infsuplemma}. To prove that $F$ is a lattice embedding, we show that $F$ is injective and that it preservers operations of taking infimum and supremum.

\medskip
For the former, assume the contrary and let $a$ and $b$ be elements of $B$ s.t. $a\nleqslant_{\mathbb{B}} b$ but $\CTo_{\varphi_{a}}\LHD\CTo_{\varphi_{b}}$. By Lemma~\ref{order2lemma}, $\CTz\vdash\forall x.\varphi_{a}(x)\rightarrow\forall x.\varphi_{b}(x)$. Hence, by definitions of $\varphi_{a}$~and~$\varphi_{b}$, $\CTz+\neg\Con_{\CTz}$ proves
$$\text{"}\{y:\ \eta(y)\}\text{ is an ultrafilter containing }\num{a}\text{"}\rightarrow\text{"}\{y:\ \eta(y)\}\text{ is an ultrafilter containing }\num{b}\text{".}$$
By the fliexibility of $\eta(y)$ and properties of Boolean algebras (provable in $\EA$), to obtain a contradiction, it is enough to show that there is an arithmetical $\Sigma_{1}$-formula~$\mu(y)$~s.t.
$$\PA\vdash\text{"}\{y:\ \mu(y)\}\text{ is an ultrafilter containing }\num{a\cap\lnot b}\text{"}$$
(notice that $a\cap\lnot b\neq\mathbb{0}$ as $a\nleqslant_{\mathbb{B}} b$). To this end, we use a formula which describes an explicit process of computing an ultrafilter on $\mathbb{B}$ which contains $a\cap\lnot b$.  Let $\nu(s, y)$ be the conjunction of the following formulae ($s_z$ denotes the $z$-th element of the sequence $s$):
\begin{itemize}
    \item $y\in B$,
    \item $s_{0}=\num{a\cap\lnot b}$,
    \item $s_{y+1}\cap y\neq\mathbb{0}$,
    \item $\forall z\leqslant y(z\in B\land s_{z}\cap z\neq\mathbb{0}\rightarrow s_{z+1}=s_{z}\cap z)$,
    \item $\forall z\leqslant y(z\in B\land s_{z}\cap z=\mathbb{0}\rightarrow s_{z+1}=s_{z}\cap \lnot z)$,
    \item $\forall z\leqslant y(z\notin B\rightarrow s_{z+1}=s_{z})$.
\end{itemize}
We put: $\mu(y):= \exists s. \nu(s,y)$. In $\PA$, the above is equivalent to a $\Sigma_{1}$-formula by the assumption that $\mathbb{B}$ has an elementary definition. Because $\CTz\vdash \PA$ (see \ref{fact_modeltheory_ct_0}), we can use induction to check that $\CTz$ proves each of the following (the details are omitted):
\begin{itemize}
    \item $\mu(\mathbb{1})\land\lnot\mu(\mathbb{0})$,
    \item $\forall y,z\in B(\mu(y)\land\mu(z)\rightarrow\mu(y\cap z))$,
    \item $\forall y,z\in B(\mu(y)\land y\leqslant_{\mathbb{B}} z\rightarrow\mu(z))$,
    \item $\forall y\in B(\mu(y)\lor\mu(\lnot y))$,
    \item $\mu(\num{a})\land\lnot\mu(\num{b})$.
\end{itemize}

\medskip
It remains to prove that $F$ preserves operations of taking infimum and supremum. Let $a$ and $b$ be elements of $B$. Then $\inf\{F(a),F(b)\}=\inf\{\CTo_{\varphi_{a}},\CTo_{\varphi_{b}}\}=\CTo_{\varphi_{a}\land\varphi_{b}}$, by definition of $F$ and Lemma \ref{infsuplemma}. Analogously, $\sup\{F(a),F(b)\}=\CTo_{\varphi_{a}\lor\varphi_{b}}$. Hence it is enough to show that $\EA\vdash\forall x(\varphi_{a}(x)\land\varphi_{b}(x)\leftrightarrow\varphi_{a\cap b}(x))$ and $\EA\vdash\forall x(\varphi_{a}(x)\lor\varphi_{b}(x)\leftrightarrow\varphi_{a\cup b}(x))$. By definitions of $\varphi_{a}$ and $\varphi_{b}$, it suffices to check that $\EA$ proves
$$\text{"}\{y:\ \eta(y)\}\text{ is an ultrafilter"}\rightarrow\big((\eta(\num{a})\land\eta(\num{b})\leftrightarrow\eta(\num{a\cap b}))\land(\eta(\num{a})\lor\eta(\num{b})\leftrightarrow\eta(\num{a\cup b}))\big)\text{.}$$
This is true because $\mathbb{B}$ is elementarily presented, so the equivalence between the second and third points in Fact \ref{fact_ultrafilters} for $\mathcal{F}=\eta$ can be proved in $\EA$.

\end{proof}


\section{Other Structural Properties}


Let us begin by introducing a definition of truth that is essentially different from any of the theories $\CTo_{\varphi}$ used in the construction of our embedding.

\begin{definition}
    Let $\Cut$ be the conjunction of $\EA$ and the following $\LT$-sentence:
    \[\forall x \bigl(\forall y<x. \CTm(y)\rightarrow \CTm(x)\bigr).\]
\end{definition}

\begin{remark}
    $\Cut$ is a definition of truth and $\CTm\vdash \Cut$. In particular $\Cut$ is conservative over $\EA$. In \cite{wcislyk_positive} it is shown that $\Cut$ is mutually definable with a theory WPT$^-$, which is a compositional theory of truth which compositional axioms are modelled after the Weak Kleene Logic.
\end{remark}
 
The following proposition witnesses that $\Cut$ is semantically conservative over $\EA$.

\begin{proposition}\label{prop_sem_com_cut}
    For every model $\mathcal{M}\models \EA$ there exists $T\subseteq M$ such that $\langle \mathcal{M}, T\rangle \models \Cut$. 
\end{proposition}
\begin{proof}[Sketch of proof]
    Fix $\mathcal{M}$. We put $T = \{\varphi\in M: \exists n\in\omega\ \ \mathcal{M}\models \Tr_n(\varphi)\}.$
\end{proof}

For the notion of recursive saturation, consult \cite{kaye}.

\begin{proposition}\label{prop_not_se_con_restr}
     If $\mathcal{M}\models \PA$ is such that for some $\CTo_{\varphi}\in\Delta$ and some $T\subseteq M$, $\langle \mathcal{M}, T\rangle \models \CTo_{\varphi}$, then $\mathcal{M}$ is recursively saturated.
\end{proposition}
\begin{proof}
    Assume $\mathcal{M}\models \PA$ and fix $\varphi$ such that $\CTo_{\varphi}\in\Delta$. By the overspill principle (see \cite{kaye}) it follows that there is a nonstandard $a\in M$ such that $\mathcal{M}\models \forall x<a.\varphi(x).$ Assume that for some $T\subseteq M$, $\langle\mathcal{M}, T\rangle\models \CTo_{\varphi}$. In particular $\langle\mathcal{M},T\rangle\models \CTm(a)$, hence $T$ is a partial nonstandard satisfaction class on $\mathcal{M}$.\footnote{More precisely $T$ encodes in a canonical way a partial nonstandard satisfaction class.} Hence $\mathcal{M}$ must be recursively saturated by Lachlan's theorem (see \cite{kaye}).
\end{proof}

\medskip
The first part of the next proposition follows directly from the observation that if $A$ is semantically conservative over $B$ and $A$ defines a theory $C$, then $C$ must also be semantically conservative over $B$ as well (we include a proof below for the sake of completeness; for an overview of various reducibility notions between theories of truth, see \cite{comparing}).

\begin{proposition}
    In $\langle \DEF, \LHD\rangle$:
    
    \begin{enumerate}
        \item no element of $\Delta$ is below $\Cut$,
        \item for every $\CTo_{\varphi}\in\Delta$ such that $\CTz\nvdash \forall x.\varphi(x)$, $\Cut$ is not below $\CTo_{\varphi}$.
    \end{enumerate}
    In particular for every $a\in\mathbb{B}$, $\CTo_{\varphi_a}$ is incomparable with $\Cut$.
\end{proposition}
\begin{proof}
For the first, suppose that there is $\CTo_{\varphi}\in\Delta$ such that $\CTo_{\varphi}\LHD\Cut$. Let $\Theta(x)$ be a definition of $T$, such that $\Cut\vdash \sigma[\CTo_{\varphi}]$, where $\sigma$ is the (unique) translation such that $\sigma(T) = \Theta$. Take any $\mathcal{M}\models \PA$ which is not recursively saturated. By Proposition \ref{prop_sem_com_cut}, let $T$ be such that $\langle\mathcal{M}, T\rangle\models \Cut$. Let $T'$ be the set defined in $\langle\mathcal{M}, T\rangle$ by $\Theta$. Then $\langle\mathcal{M}, T'\rangle\models \CTo_{\varphi}$, what contradicts Proposition \ref{prop_not_se_con_restr}.

\medskip
For the second, we use an analogous argument to that presented in \ref{order2lemma}. This is possible since $\Cut + \IND(\LT)\vdash\CT$. Aiming at a contradiction, assume that for some $\varphi$, such that $\CTz\nvdash\forall x.\varphi(x)$, $\Cut\LHD\CTo_{\varphi}$. Take $\langle \mathcal{M}, T\rangle\models \CTz + \exists x.\neg\varphi(x) + \neg\Con_{\CTz + \exists x.\neg\varphi(x)}.$ For the least
$a$ such that $\mathcal{M}\models \neg\varphi(a)$, we see that $\langle\mathcal{M}, T_a\rangle\models \CTo_{\varphi} + \IND(\LT)$. Hence, by our assumption, there is $T'$ such that $\langle \mathcal{M}, T'\rangle \models \Cut + \IND(\LT)$. However, then $\mathcal{M}\models \Con_{\CTz + \exists x.\neg\varphi(x)},$ a contradiction.
\end{proof}

Let us put $\Delta\vee \Cut:= \{\CTo_{\varphi}\vee\Cut:\ \CTo_{\varphi}\in\Delta\}$. Then we can reprove Main Theorem for $\Delta\vee \Cut$ instead of $\Delta$. The next theorem provides us with another way of realizing the embedding of $\mathbb{B}$ into c$\DEF$. This time the image of the embedding consists uniquely of theories $\LHD$-below $\Cut$.

\begin{theorem}
$\Delta\vee\Cut$ is a countable universal distributive lattice.
\end{theorem}
\begin{proof}[Sketch of proof]
It is easy to see that $\Delta\vee\Cut$ is indeed a sublattice of c$\DEF$. It is also easy to see that the function
$$a\longmapsto\CTo_{\varphi_{a}}\lor\ \Cut\text{,}$$
where $\varphi_a$ is defined as in Section \ref{sec_main}, preserves suprema and infima as the function $F$ from Section \ref{sec_main} preserves them. Hence it is enough to show that this function is injective. Equivalently, for $a\neq\mathbb{0}$, it is not true that
$$\CTo_{\varphi_{a}}\lor\ \Cut\RLHD\CTo_{\varphi_{\mathbb{0}}}\lor\ \Cut\text{.}$$
We prove that $\CTo_{\varphi_{a}}\lor\ \Cut\not\LHD\CTo_{\varphi_{\mathbb{0}}}$. We begin by constructing a model of
$$\CTz+\forall x.\varphi_{a}+\lnot\forall x.\varphi_{\mathbb{0}}=\CTz+\forall x.\varphi_{a}+\lnot\Con_{\CTz}\text{.}$$
Then proceed similarly as in the proof of Lemma \ref{order2lemma} obtaining a model $\langle\mathcal{M},T_{c}\rangle\models \forall x.\varphi_{a}+\text{IND}(\LT)+\lnot\Con_{\CTz}$. It cannot define $\CTo_{\varphi_a}\lor\ \Cut$ as $\mathcal{M}$ cannot be expanded to a model of $\CT$.  

\end{proof}

Unfortunately, we still do not know if there is a truth definition for $\LA$ that does not define $\UTBm$. We get the following small insight, which crucially relies on the observation due to Volker Halbach (see \cite{halbach_disq_anal}, Proposition 9.2).

\begin{proposition}\label{prop_between_tb_utb}
    If $\alpha$ is a definition of truth for $\LA$, then there is a true $\Pi_2$-sentence $\pi$, such that $\UTB\LHD \alpha + \pi + \IND(\mathcal{L}_{\alpha}).$
\end{proposition}
\begin{proof}
    Fix $\alpha\in\DEF$ and let $\pi$ be the $\Pi_2$ $\LA$-sentence saying ``$\alpha$ defines $\TBm$ by the use of $\Theta(x)$'', where $\Theta(x)$ witnesses the definability of $\TBm$ in $\alpha$. By the Kreisel-Levy theorem (see \cite{bek05}) we conclude that the uniform reflection principle over $\alpha$ is provable in $\alpha+\IND(\mathcal{L}_{\alpha})$. That is, for every $\mathcal{L}_{\alpha}$-formula $\varphi(x)$ the following sentence is provable in $\alpha+\IND(\mathcal{L}_{\alpha})$
    \[\forall x \bigl(\Prov_{\alpha}[\varphi(\dot{x})]\rightarrow \varphi(x)\bigr).\]
    Fix any $\LA$ formula $\psi(x)$ and observe that
    \[\pi + \EA\vdash \forall x.\Prov_{\alpha}[\Theta[\psi(\dot{\dot{x}})]\leftrightarrow \psi(\dot{x})].\]
    Hence, by uniform reflection, we obtain
    \[\pi+\IND(\mathcal{L}_{\alpha}) + \alpha \vdash \forall x \bigl(\Theta[\psi(\dot{x})]\leftrightarrow \psi(x)\bigr).\]
\end{proof}

The corollary below witnesses that the canonical compositional definition of truth is optimal up to the quantifier complexity. The core of the argument (in the proof of Theorem~\ref{tw_complexity}) follows from the well-known property of \textit{strict} $\Sigma^1_1$ formulae. Since we would like to make the paper self-contained we present also a proof which does not use a detour through subsystems of Second Order Arithmetic. Both reasonings are essentially the same: the key role in the first one is played by the Weak K\"onig's Lemma, while the second uses the arithmetized completeness theorem. We recall the definition and theorem from \cite{simpson}, Section VIII.2

\begin{definition}
    A formula $\phi$ is strict $\Sigma_1^1$ ($\phi\in S\Sigma^1_1$) if  $\phi$ is of the form $\exists X\psi$, for some $\psi\in\Pi^0_1$.
\end{definition}

\begin{theorem}\label{thm_strict_wkl}
    Let $\mathcal{M}'$ be a model of WKL$_0$ and $\mathcal{M}$ be an arbitrary $\omega$-submodel of $\mathcal{M}'$ (that is the universe of first-order objects of $\mathcal{M}$ and $\mathcal{M}'$ are the same). Then the following conditions are quivalent
    \begin{enumerate}
        \item  $\mathcal{M}$ is a model of WKL$_0$.
        \item $\mathcal{M}$ and $\mathcal{M}'$ satisfy the same strict $\Sigma^1_1$ formulae (with parameters from $\mathcal{M}$).
    \end{enumerate}
\end{theorem}

\medskip
We say that a sentence $\alpha$ has a standard model if the standard model of arithmetic, $\N$, can be expanded to a model of $\alpha$. That is, in $\N$ one can find interpretations for additional symbols used in $\alpha$ in such a way that the expanded structure makes $\alpha$ true. Without loss of generality, we assume that apart from the symbols for the arithmetical operations, sentences that we consider use a single unary predicate $P$ (this can be assumed thanks to the definable pairing function in $\EA$).

\begin{theorem}\label{tw_complexity}
    If $\alpha$ is a $\Sigma_2$-sentence such that
    \begin{itemize}
        \item $\alpha$ has a standard model and
        \item the only non-logical, non-arithmetic symbol that occurs in $\alpha$ is a unary relation symbol~$P$,
    \end{itemize}
    then there is an arithmetically definable set $A$ such that $(\mathbb{N}, A)\models \alpha$.
\end{theorem}

\begin{proof}[First proof]
    Fix $\alpha:= \exists x. \beta(x)$, where $\beta$ is $\Pi_1$, and assume that $B$ is such that $(\mathbb{N}, B)\models \alpha$. Fix $n$ such that $(\mathbb{N}, B)\models \beta(n)$. Since $n$ can be named with a standard numeral we will stop mentioning it. Let $\mathcal{X}$ be the least set containing $B$, closed under the Turing Jumps and Turing Reducibility. Then $(\mathbb{N}, \mathcal{X})\models \textnormal{ACA}_0$. Since $B\in\mathcal{X}$, $(\mathbb{N},\mathcal{X})\models \exists X.\beta[t\in X/P(t)]$ (that is: each occurrence of $P(t)$ for a term $t$ is replaced in $\beta$ with $t\in X$). Since $\beta[t\in X/P(t)]$ is $\Pi^0_1$, $\exists X.\beta[t\in X/P(t)]$ is a strict $\Sigma^1_1$ formula. Therefore, since $\textnormal{Def}(\mathbb{N})\subseteq \mathcal{X}$, by Theorem \ref{thm_strict_wkl}, $(\mathbb{N}, \textnormal{Def}(\mathbb{N}))\models \exists X. \beta[t\in X/P(t)],$ where $\textnormal{Def}(\mathbb{N})$ is the collection of arithmetically definable subsets of $\mathbb{N}$. Hence, there is an arithmetically definable subset $A$ such that $(\mathbb{N},A)\models \alpha$.
\end{proof}

\medskip
The second proof of the theorem uses, instead of the Weak K\"onig's Lemma, some standard techniques related to the arithmetized completeness theorem. We briefly recall this material, redirecting to \cite{lelyk_global}, Section 2.3 for the details. We say that a model $\mathcal{N}$ is \emph{strongly interpretable} in a model $\mathcal{M}$ if:
\begin{enumerate}
    \item the domain of $\mathcal{N}$ and the functions and relations of $\mathcal{N}$ are definable subsets of $\mathcal{M}$;
    \item there is an $\mathcal{M}$-definable relation $S_{\mathcal{N}}$ such that in $\mathcal{M}$ it is true that $S_{\mathcal{N}}$ is a satisfaction relation for $\mathcal{N}$.
\end{enumerate}
If $\mathcal{N}\models \PAm$ is strongly interpretable in $\mathcal{M}\models\PA$, then there is an $\mathcal{M}$-definable initial embedding $\iota:M\rightarrow N$, which makes $\mathcal{M}$ an initial segment of $\mathcal{N}$ (if $b\in M$ and $\mathcal{N}\models a<b$, then $a\in M$). Finally, if $\mathcal{M}\models \PA + \Con(Th)$, then there is a model $\mathcal{N}\models Th$, which is strongly interpretable in $\mathcal{M}$.

\begin{proof}[Second proof of Theorem \ref{tw_complexity}]
Fix any $\alpha: = \exists  x .\beta(x)$, where $\beta$ is $\Pi_1$, and suppose it satisfies the assumptions.  Since $\alpha$ has a standard model, there is an $n\in\omega$ such that $\beta(\num{n})$ has a standard model as well.  By our assumptions $\mathbb{N}\models\Con(\qcr{\beta(\num{n})})$. By the arithmetized completeness theorem, there is $\mathcal{M}\models \beta(\num{n})$ which is strongly interpretable in $\mathbb{N}$. Let  $\iota$ be the embedding of $\mathbb{N}$ onto an initial segment of $\mathcal{M}$, $P^{\mathcal{M}}$ be the $\mathbb{N}$-definable interpretation for $P$ in $\mathcal{M}$ and $\iota[\mathbb{N}]$ be the image of $\iota$ (which is $\mathbb{N}$-definable as well). Define $A:= \iota[\mathbb{N}]\cap P^{\mathcal{M}}$. Then $A$ is definable in $\mathbb{N}$ and $(\mathbb{N}, A)$ is a submodel of $\mathcal{M}$. Also, since $\mathbb{N}$ is an initial segment of $\mathcal{M}$ closed under multiplication and $\beta(\num{n})$ is a $\Pi_1$-formula, then $(\mathbb{N}, A)\models \beta(\num{n})$.
\end{proof}

\begin{corollary}
    If $\alpha$ satisfies the assumptions of Theorem \ref{tw_complexity}, then $\Th(\mathbb{N})+\alpha$ is not a theory of truth for $\LA$.
\end{corollary}


\section{Open Questions}


Despite our efforts, the main question motivating our investigations remains open. We state it once again for future reference:
\paragraph{Question 1} Is there $\alpha\in\DEF$ that does not define $\UTBm$?

\bigskip

In Theorem \ref{tw_pakh-viss} we proved that $\DEF$ is not $\Sigma_2$-definable. The natural definition of this set is $\Sigma_3$, so it is natural to ask:

\paragraph{Question 2} Is $\DEF$ a $\Pi_3$-definable set?
\bigskip

Additionally we would like to know more about the structure of $(\DEF, \LHD)$. For two definitions of truth $\alpha,\ \beta$, let us put $[\alpha,\beta]:= \{\gamma:\ \alpha\LHD\gamma\LHD\beta\}$.

\paragraph{Question 3} Is it true that for all $\alpha\LHD\beta$, if the interval $[\alpha, \beta]$ is finite, then $\beta\LHD\alpha$?

\paragraph{Question 4} Are there $\alpha,\beta\in \cDEF$ such that the sublattice $\langle [\alpha,\beta], \LHD\rangle\subseteq \langle\DEF,\LHD\rangle$ has the structure of the atomless Boolean algebra? In particular, is it true for $\alpha = \Cut$ and $\beta = \CTm$? Does the answer change if we drop the requirement that the algebra is atomless?

\bigskip

The last question is motivated by the typical uses of the truth predicate in obtaining finite axiomatizations of infinite theories. See \cite{truth_feasible} for relevant context and definitions.

\paragraph{Question 5} Is there a definition of truth $\alpha$ (with a formula $\Theta(x)$ as an interpretation of the truth predicate) such that $\alpha + \forall x\bigl(\textnormal{Ax}_{\PA}(x)\rightarrow \Theta(x)\bigr)$ has at most polynomial speed-up over $\PA$? $\textnormal{Ax}_{\PA}(x)$ is the canonical formula which strongly represents the set of axioms of $\PA$ in $\EA$.


\bibliographystyle{plain}
\bibliography{sonata.bib}

\appendix
\appendixpage

\section{Characterizations of definability in terms of provability}

Now we return to the problem of unifying our methods of showing comparability (Lemma \ref{lem_comp}) and incomparability (Lemma \ref{order2lemma}) of theories of the form $\CTo_{\alpha}$, with respect to the relation of definability. We note that there is indeed a significant logical gap between the two lemmata: the sufficient condition for the definability  is given in the comparability lemma in terms of provability in $\EA$, while the necessary condition in the incomparability lemma uses a much stronger theory, $\CT_0$. The arithmetical consequences of the latter theory of truth are exhausted only by the union of all finite iterations of uniform reflection principles over $\PA$, the theory usually denoted by $\textnormal{RFN}^{<\omega}(\PA)$. A natural conjecture would be that the condition mentioned in the remark immediately following the comparability lemma is not only sufficient, but also necessary, i.e. we have the following equivalence for any $\alpha, \beta$ such that $\CTo_{\alpha}, \CTo_{\beta}\in\Delta$:
\begin{equation}\label{app_conj}\tag{$*$}
\bigl(\exists n\in\omega\ \EA\vdash \forall x (\alpha(x+\underline{n})\rightarrow\beta(x))\bigr) \iff \CTo_{\alpha}\LHD\CTo_{\beta}.
\end{equation}
As we already remarked, the left-to-right implication indeed holds. For the converse one, one could try the following strategy: assume that $\CTo_{\alpha}\LHD\CTo_{\beta}$ but for every $n\in\omega$ $\EA+\exists x (\alpha(x+n)\wedge\neg \beta(x))$ is consistent. Then by a straightforward compactness argument, one can show that there is a model $\mathcal{M}\models \EA$, in which for some number $c>\omega$, the length of the initial segment determined by $\alpha$ is at least by $c$ greater than the length of the initial segment determined by $\beta$. The same relation between $\alpha$ and $\beta$ holds in any elementary extension of $\mathcal{M}$. However, there is $\mathcal{N}$, an elementary extension of $\mathcal{M}$, in which there is a subset $T\subseteq N$ such that $(\mathcal{N},T)\models \CTm$ (see e.g. \cite{enayat_visser_new}). Consider $T_{\beta}:= \{\sigma\in N: \sigma\in T\wedge \mathcal{N}\models \beta(\dpt(\sigma))\}$. Then $(\mathcal{N}, T_{\beta})\models \CTo_{\beta}$. So there is a formula $\Theta(x)\in\LT$ such that $(\mathcal{N}, T_{\beta})\models \CTo_{\alpha}[\Theta(x)/T(x)]$. In the previous sentence, and below, for $\Phi(x)\in\LT$ and any formu $\Psi(x)$, \[\Phi[\Psi(t)/T(t)]\]
denotes the result if a substitution of a formula $\Psi(t)$ for any occurrence of a formula $T(t)$, possibly preceded by renaming the bounded variables. The last conclusion may seem contradictory, since e.g. in the standard model one can define partial truth predicates, but the truth predicate for all formulae is not definable. However, this intuition breaks down in the case of nonstadard models and truth predicates that do not have to satisfy induction axioms\footnote{We are grateful to Bartosz Wcisło for pointing out this application of Smith's theorem.}.

\begin{proposition}[Wcisło]
    Assume that $\mathcal{M}$ is a countable and recursively saturated model of $\PA$ and $c\in M$ is a nonstadard element which is definable without parameters in $\mathcal{M}$. Then, there is $T\subseteq M$ such that $(\mathcal{M}, T)\models \CTm(c)\wedge \forall x \bigl(T(x)\rightarrow \dpt(x)\leqslant c\bigr)$ and for some formula $\Theta(x)$, $(\mathcal{M}, T)\models \CTm[\Theta(x)/T(x)]$.
\end{proposition}
In other words, for an arbitrary countable recursively satrurated model and for an arbitrary nonstandard definable $c$, one can find a truth predicate for sentences of depth at most $c$, which can define a full truth predicate for all sentences in the sense of $\mathcal{M}$.

\begin{proof}[Sketch of proof]
Fix $\mathcal{M}$ and $c$ as in the assumptions.  Let $\varepsilon_k(x)$ be a formula defined inductively
\begin{align*}
    \varepsilon_0(x)&:= x=x\\
    \varepsilon_{k+1}(x)&:= \varepsilon_k(x)\vee\varepsilon_k(x)
\end{align*}
Thus the syntactic tree of $\varepsilon_k(x)$ is a full binary tree of depth $k$ with internal vertices and root labelled by disjunctions and with the formula $x=x$ in the leaves. Let $T'\subseteq M$ be such that $(\mathcal{M},T')\models \CTm$ and $(\mathcal{M}, T')$ is recursively saturated (such a $T'$ exists by the conservativity of $\CTm$ over $\PA$ and the chronic resplendency of countable recursively saturated models). Then, by Theorem 3.3 of \cite{smith_nonstandard}, there is a $T''\subseteq M$ such that $(\mathcal{M}, T'')\models \CTm$  and\footnote{Smith's theorem is stated for a satisfaction class, but it also works for models of $\CTm$.}
\[T':= \{a\in M : (\mathcal{M}, T'')\models T''[\varepsilon_c(\dot{a})]\}.\]
Observe now that $\varepsilon_c(x)$ is a formula of depth $c$ and it is definable without parameters (since it is definable from $c$). Hence, for $T:= \{\sigma\in M : \sigma\in T''\wedge \mathcal{M}\models \dpt(\sigma)\leqslant c\}$, set $T'$ is definable without parameters in $(\mathcal{M}, T)\models \CTm(c)\wedge \forall x \bigl(T(x)\rightarrow \dpt(x)\leqslant c\bigr).$
\end{proof}

Obviously, the above example is not sufficient to disprove conjecture \eqref{app_conj}. However, we see no way to fix the proof strategy described above.

\medskip
Having said that, we note that if we assume that our definitions of truth extend $\PA$ instead of $\EA$ (i.e. are "finitely axiomatisable modulo $\PA$"), then we have the following version of the incomparability lemma. To motivate the formulation of the thesis, let us observe that there is a constant $k$ such that for all $n>0$, $\Tr_n(x)$ (see "The arithmetization of syntax and metalogic" in the preliminaries) has depth at most $k\cdot n$. Let $k_0$ be the least such $k$.

\begin{proposition}\label{prop_conj_PA}
For all $\alpha$, $\beta$ such that $\CTo_{\alpha}, \CTo_{\beta}\in\Delta$,
\[(\CTo_{\alpha}+\PA)\LHD(\CTo_{\beta}+\PA)\Longrightarrow\bigl(\exists n\in\omega\ \PA\vdash \forall x (\alpha(k_0\cdot x+n)\rightarrow\beta(x))\bigr)\text{.}\]
\end{proposition}
\begin{proof}[Sketch of proof]
     In the proof we deal with two satisfaction relations for models that are strongly interpretable in a fixed model $\mathcal{M}\models\PA$. For a model $\mathcal{K}$ strongly interpretable in $\mathcal{M}$, $\mathcal{K}\models \phi$ denotes the satisfiability of $\phi$ by $\mathcal{K}$ according to the standard satisfaction relation (as defined outside of $\mathcal{M}$), and $\mathcal{K}\models^{\mathcal{M}} \phi$ denotes the satisfiability according to the $\mathcal{M}$-definable satisfaction relation. We observe that whenever $\phi$ is a standard sentence and $\mathcal{K}$ is any model which is strongly interpretable in $\mathcal{M}$, then $\mathcal{K}\models \phi$ iff $\mathcal{K}\models^{\mathcal{M}}\phi$ ("satisfaction is absolute"). We now pass to the argument.

     \medskip
     As previously, suppose that $(\CTo_{\alpha}+\PA)\LHD(\CTo_{\beta}+\PA)$ and that there exists a nonstandard model $\mathcal{M}\models \PA$ with nonstandard numbers $a,c\in M$ such that $\mathcal{M}\models \alpha(k_0\cdot a+c)\wedge \neg\beta(a)$. Let $\Phi$ be such that $\CTo_{\beta}+\PA\vdash \CTo_{\alpha}[\Phi(x)/T(x)]$. Since both $\alpha$ and $\beta$ define initial segments, this means that in $\mathcal{M}$ the length of the initial segment defined by $\alpha$ is much greater than the length of the initial segment defined by $\beta$. Let $n$ be such that both $\alpha$ and $\beta$ are of depth at most $n$ and let $k$ be the depth of $\Phi$. By the reflexivity of $\PA$ and an easy overspill argument there is a nonstandard number $d$ such that in $\mathcal{M}$ there is no proof of $0=1$ from the axioms of $\ISd$ and true arithmetical sentences of depth at most $n$, i.e.
    \[\mathcal{M}\models \Con(\ISd+\Tr_{n}).\]
    By the arithmetized completeness theorem, there is a strongly interpretable in $\mathcal{M}$ model $\mathcal{N}$ such that $\mathcal{N}\models^{\mathcal{M}}\ISd+\Tr_{n}$. 
    In particular, $\mathcal{M}$ is ($\mathcal{M}$-definably isomorphic to) an $n$-elementary initial segment of $\mathcal{N}$, hence we have $\mathcal{N}\models\alpha(k_0\cdot a+c)\wedge \neg\beta(a)$. Moreover, from an external perspective, $\mathcal{N}\models \PA$ (but $\mathcal{M}$ may not know this).

    \medskip
     We now work in $\mathcal{M}$.  Let $e$ be the greatest element such that $\mathcal{N}\models^{\mathcal{M}} \beta(e)$. Then $e\leqslant a$. For any $y\in M$, a formula $\Tr_{y}(x)$ is a partial truth predicate for sentences of complexity $y$ in the sense of $\mathcal{N}$. Even though for nonstandard $y$, $\Tr_y$ is a nonstandard formula, we can make sense of it using $\models^{\mathcal{M}}$. Observe that $\mathcal{N}\models^{\mathcal{M}} \CTo_{\beta}[\Tr_{e}(x)/T(x)]$. Put $T_e:= \{x\in N : \mathcal{N}\models^{\mathcal{M}} \Tr_e(x)\}$. Since $T_e$ is a definable subset of $\mathcal{N}$, there is an $\mathcal{M}$-definable satisfaction relation for the model $(\mathcal{N}, T_e)$. Seen from the external (to $\mathcal{M}$) perspective, $(\mathcal{N}, T_e)\models \CTo_{\beta}+\PA$, hence also it is true that $(\mathcal{N}, T_e)\models \CTo_{\alpha}[\Phi(x)/T(x)]$. Since $\CTo_{\alpha}[\Phi(x)/T(x)]$ is a standard sentence, $(\mathcal{N}, T_e)\models^{\mathcal{M}} \CTo_{\alpha}[\Phi(x)/T(x)]$. Let $\Psi:= \Phi[\Tr_e(x)/T(x)]$. Then $\Psi$ is an arithmetical formula (in the sense of $\mathcal{M}$) of depth at most $k_0\cdot e+k$.
    It follows that $\mathcal{N}\models^{\mathcal{M}} \CT^-(k_0\cdot a+c)[\Psi(x)/T(x)]$, so in $\mathcal{N}$, $\Psi$ is a truth predicate for sentences of complexity $k_0\cdot a+c$. Let $\lambda$ be the $\Psi$-liar sentence, i.e. a sentence constructed as in the standard proof of Tarski's theorem on the undefinability, such that $\mathcal{N}\models^{\mathcal{M}}\lambda \leftrightarrow \neg\Psi(\qcr{\lambda})$. By the inspection of the usual proof of the fixed point lemma, $\lambda$ is of depth $k_0\cdot e +l$ for some standard number $l$. In particular $\mathcal{N}\models^{\mathcal{M}}\lambda\leftrightarrow \Psi(\qcr{\lambda})$. A contradiction.
\end{proof}

What makes the above proof possible is the reflexivity of $\PA$. We use it to show that every model of $\PA$ strongly interprets a suitable model of $\PA$.

\section{The lattice structure of the definability order}\label{appendix_supinf}

\thmsupinf*
\begin{proof}
The proof has three parts. First, we show that $\alpha\lor\beta$ and $\alpha\land\repsymb_{\alpha}(\beta)$ are syntactically conservative over $\EA$ definitions of truth (whenever $\alpha$ and $\beta$ are such). Then we show that $\alpha\lor\beta$ is an infimum and $\alpha\land\repsymb_{\alpha}(\beta)$ is a supremum of $\alpha$ and $\beta$ (the proof also shows that the operations are well defined as operations on the equivalence classes of truth definitions modulo the relation of mutual definability, i.e. if $\alpha\RLHD\alpha'$ and $\beta\RLHD\beta'$, then $\alpha\lor\beta\RLHD\alpha'\lor\beta'$ and $\alpha\land\repsymb_{\alpha}(\beta)\RLHD\alpha'\lor\repsymb_{\alpha'}(\beta')$). We complete the proof by showing that the structure satisfies distributivity axioms.

\medskip
Both $\alpha\lor\beta$ and $\alpha\land\repsymb_{\alpha}(\beta)$ extend $\EA$. Given a sentence $\sigma\in\LA$, suppose that $\alpha\lor\beta\vdash\sigma$, then in particular $\alpha\vdash\sigma$, so $\EA\vdash\sigma$ by the conservativity of $\alpha$. Now suppose that $\alpha\land\repsymb_{\alpha}(\beta)\vdash\sigma$, then
$$\emptyset\vdash\alpha\rightarrow(\repsymb_{\alpha}(\beta)\rightarrow\sigma)\text{.}$$
Since the sets of the symbols present in $\alpha$ and $\repsymb_{\alpha}(\beta)$ respectively, intersect exactly at the arithmetic symbols, by Craig's interpolation theorem (see 2.5.5 in \cite{Handbook_Prf_Th}), there exists an arithmetical sentence $\tau$ such that
$$\emptyset\vdash\alpha\rightarrow\tau\text{ \ \ and \ \ }\emptyset\vdash\tau\rightarrow(\repsymb_{\alpha}(\beta)\rightarrow\sigma)\text{,}$$
and hence $\emptyset\vdash\repsymb_{\alpha}(\beta)\rightarrow(\tau\rightarrow\sigma)$. The arithmetical consequences of $\alpha$ and $\repsymb_{\alpha}(\beta)$ are exactly the same, so $\emptyset\vdash\repsymb_{\alpha}(\beta)\rightarrow\tau$, and therefore $\repsymb_{\alpha}(\beta)\vdash\sigma$. Hence $\EA\vdash\sigma$, since $\repsymb_{\alpha}(\beta)$ is syntactically conservative over $\EA$.

\medskip
Now suppose that $\Theta_{\alpha}(x)$ and $\Theta_{\beta}(x)$ act as truth-predicates in $\alpha$ and $\beta$ respectively. Then $\Theta_{\alpha}(x)$ also acts as a truth-predicate in $\alpha\land\repsymb_{\alpha}(\beta)$. As a truth-predicate in $\alpha\lor\beta$ we can use
$$(\alpha\rightarrow\Theta_{\alpha}(x))\land(\lnot\alpha\rightarrow\Theta_{\beta}(x))\text{.}$$

We now come to the second part of the proof. We see that both $\alpha$ and $\beta$ define $\alpha\lor\beta$, as each of them proves it. Similarly, $\alpha\land\repsymb_{\alpha}(\beta)\RHD\alpha\text{, }\beta$. Now assume that $\alpha\RHD\gamma$ and $\beta\RHD\gamma$. Suppose that the former definability is given by the translation of the form $\mathcal{L}_{\gamma}\ni\mathfrak{R}(\overline{x})\longmapsto\varphi_{\mathfrak{R}}(\overline{x})\in\mathcal{L}_{\alpha}$ and that the second is given by the translation of the form $\mathcal{L}_{\gamma}\ni\mathfrak{R}(\overline{x})\longmapsto\psi_{\mathfrak{R}}(\overline{x})\in\mathcal{L}_{\beta}$. Then "$\alpha\lor\beta\RHD\gamma$" is given by the translation of the form
$$\mathfrak{R}(\overline{x})\longmapsto(\alpha\rightarrow\varphi_{\mathfrak{R}}(\overline{x}))\land(\lnot\alpha\rightarrow\psi_{\mathfrak{R}}(\overline{x}))\text{.}$$
Now assume that $\delta\RHD\alpha$ and $\delta\RHD\beta$. Suppose that the first is given by the translation of the form $\mathcal{L}_{\alpha}\ni\mathfrak{R}(\overline{x})\longmapsto\varphi_{\mathfrak{R}}(\overline{x})\in\mathcal{L}_{\delta}$ and that the second is given by the translation of the form $\mathcal{L}_{\beta}\ni\mathfrak{R}(\overline{x})\longmapsto\psi_{\mathfrak{R}}(\overline{x})\in\mathcal{L}_{\delta}$. Then "$\delta\RHD\alpha\land\repsymb_{\alpha}(\beta)$" is given by the translation of the form:
\begin{itemize}
    \item $\mathfrak{R}(\overline{x})\longmapsto\varphi_{\mathfrak{R}}(\overline{x})$, for $\mathfrak{R}\in\mathcal{L}_{\alpha}$,
    \item $\mathfrak{R}(\overline{x})\longmapsto\psi_{\repsymb_{\alpha}^{-1}(\mathfrak{R})}(\overline{x})$, for $\mathfrak{R}\in\mathcal{L}_{\repsymb_{\alpha}(\beta)}$.
\end{itemize}

Now we come to the third part of the proof. First, we see that "$x\cap(y\cup z)=(x\cap y)\cup(x\cap z)$" holds: given definitions of truth $\alpha$, $\beta$ and $\gamma$,
$$\alpha\land\repsymb_{\alpha}(\beta\lor\gamma)=(\alpha\land\repsymb_{\alpha}(\beta))\lor(\alpha\land\repsymb_{\alpha}(\gamma))\text{.}$$
Now we show that "$x\cup(y\cap z)=(x\cup y)\cap(x\cup z)$" holds, i.e. given definitions of truth $\alpha$, $\beta$~and~$\gamma$,
$$\alpha\lor(\beta\land\repsymb_{\beta}(\gamma))\RLHD(\alpha\lor\beta)\land\repsymb_{\alpha\lor\beta}(\alpha\lor\gamma)\text{.}$$
The definability $\alpha\lor(\beta\land\repsymb_{\beta}(\gamma))\RHD(\alpha\lor\beta)\land\repsymb_{\alpha\lor\beta}(\alpha\lor\gamma)$ is given by the translation of the form:
\begin{itemize}
    \item $\mathfrak{R}(\overline{x})\longmapsto\mathfrak{R}(\overline{x})$, for $\mathfrak{R}\in\mathcal{L}_{\alpha}\cup\mathcal{L}_{\beta}$,
    \item $\mathfrak{R}(\overline{x})\longmapsto\repsymb_{\alpha\lor\beta}^{-1}(\mathfrak{R})(\overline{x})$, for $\mathfrak{R}\in\mathcal{L}_{\repsymb_{\alpha\lor\beta}(\alpha)}\setminus\mathcal{L}_{\repsymb_{\alpha\lor\beta}(\gamma)}$,
    \item $\mathfrak{R}(\overline{x})\longmapsto\repsymb_{\beta}(\repsymb_{\alpha\lor\beta}^{-1}(\mathfrak{R}))(\overline{x})$, for $\mathfrak{R}\in\mathcal{L}_{\repsymb_{\alpha\lor\beta}(\gamma)}\setminus\mathcal{L}_{\repsymb_{\alpha\lor\beta}(\alpha)}$,
    \item $\mathfrak{R}(\overline{x})\longmapsto(\alpha\rightarrow\repsymb_{\alpha\lor\beta}^{-1}(\mathfrak{R})(\overline{x}))\land(\lnot\alpha\rightarrow\repsymb_{\beta}(\repsymb_{\alpha\lor\beta}^{-1}(\mathfrak{R}))(\overline{x}))$, for $\mathfrak{R}\in\mathcal{L}_{\repsymb_{\alpha\lor\beta}(\alpha)}\cap\mathcal{L}_{\repsymb_{\alpha\lor\beta}(\gamma)}$.
\end{itemize}
We complete the proof claiming that the definability $(\alpha\lor\beta)\land\repsymb_{\alpha\lor\beta}(\alpha\lor\gamma))\RHD\alpha\lor(\beta\land\repsymb_{\beta}(\gamma))$ is given by the translation which maps $\mathfrak{R}(\overline{x})$ to:
\begin{itemize}
    \item $\mathfrak{R}(\overline{x})$, for $\mathfrak{R}\in\mathcal{L}_{\beta}\setminus\mathcal{L}_{\alpha}$,
    \item $\bigl[\repsymb_{\alpha\lor\beta}(\alpha)\rightarrow\repsymb_{\alpha\lor\beta}(\mathfrak{R})(\overline{x})\bigr]\land\bigl[\lnot\repsymb_{\alpha\lor\beta}(\alpha)\rightarrow\bigl((\alpha\rightarrow\mathfrak{R}(\overline{x}))\land(\lnot\alpha\rightarrow\repsymb_{\alpha\lor\beta}(\repsymb_{\beta}^{-1}(\mathfrak{R}))(\overline{x}))\bigr)\bigr]$, for~$\mathfrak{R}\in\mathcal{L}_{\alpha}\cap\mathcal{L}_{\repsymb_{\beta}(\gamma)}$,
    \item $(\repsymb_{\alpha\lor\beta}(\alpha)\rightarrow\repsymb_{\alpha\lor\beta}(\mathfrak{R})(\overline{x}))\land(\lnot\repsymb_{\alpha\lor\beta}(\alpha)\rightarrow\mathfrak{R}(\overline{x}))$, for $\mathfrak{R}\in\mathcal{L}_{\alpha}\setminus\mathcal{L}_{\repsymb_{\beta}(\gamma)}$,
    \item $\repsymb_{\alpha\lor\beta}(\repsymb_{\beta}^{-1}(\mathfrak{R}))(\overline{x})$, for $\mathfrak{R}\in\mathcal{L}_{\repsymb_{\beta}(\gamma)}\setminus\mathcal{L}_{\alpha}$.
\end{itemize}
Notice that $\mathcal{L}_{\alpha}\cup\mathcal{L}_{\beta}\cup\mathcal{L}_{\repsymb_{\beta}(\gamma)}=(\mathcal{L}_{\beta}\setminus\mathcal{L}_{\alpha})\oplus(\mathcal{L}_{\alpha}\cap\mathcal{L}_{\repsymb_{\beta}(\gamma)})\oplus(\mathcal{L}_{\alpha}\setminus\mathcal{L}_{\repsymb_{\beta}(\gamma)})\oplus(\mathcal{L}_{\repsymb_{\beta}(\gamma)}\setminus\mathcal{L}_{\alpha})$ as $\mathcal{L}_{\alpha}\setminus\mathcal{L}_{\repsymb_{\beta}(\gamma)}\supseteq\mathcal{L}_{\alpha}\cap\mathcal{L}_{\beta}$.

\end{proof}


\end{document}